\documentclass[12pt, reqno]{amsart}
\usepackage{amsmath, amsthm, amscd, amsfonts, amssymb, graphicx, xcolor, slashed}
\usepackage[bookmarksnumbered, colorlinks, plainpages]{hyperref}
\usepackage{amsaddr}
\usepackage[numbers,sort]{natbib}
\newcommand{\pa}{\partial}
\newcommand{\ins}{\int_{\mathbb{S}^2}}
\newtheorem{theorem}{Theorem}[section]
\newtheorem{lemma}[theorem]{Lemma}
\newtheorem{proposition}[theorem]{Proposition}
\newtheorem{corollary}[theorem]{Corollary}
\theoremstyle{definition}

\theoremstyle{remark}
\newtheorem{remark}[theorem]{Remark}
\numberwithin{equation}{section}
\begin{document}
\setcounter{page}{1}

\centerline{}

\centerline{}

\title[]{ Non-time-decaying global classical solutions to nonlinear wave equations in 3D under the null condition}
\author[ ]{Zexian Zhang  }
\address{ Shanghai Center for Mathematical Sciences, Fudan University, Shanghai, China. \\
E-mail:  {\rm\textcolor[rgb]{0.00,0.00,0.84}{23110840019@m.fudan.edu.cn}}}
%\email{\textcolor[rgb]{0.00,0.00,0.84}{23110840019@m.fudan.edu.cn}}
\author[ ]{ Yi Zhou$^{\ast}$ }

\address{ School of Mathematics Science, Fudan University, Shanghai, China.\\
E-mail: {\rm\textcolor[rgb]{0.00,0.00,0.84}{yizhou@fudan.edu.cn}}}
\thanks{$^{\ast}$Corresponding author}
%\email{\textcolor[rgb]{0.00,0.00,0.84}{yizhou@fudan.edu.cn}}
%\dedicatory{This paper is dedicated to Professor ABCD}

%\subjclass[2010]{Primary 46L55; Secondary 44B20.}

%\keywords{Analysis, PDEs, algebra, number theory, applied mathematics}
 
%\date{Received: xxxxxx; Revised: yyyyyy; Accepted: zzzzzz.
%\newline \indent $^{*}$ Corresponding author}

\begin{abstract}
  We present an alternative proof of the global well-posedness of nonlinear wave equations in three spatial dimensions under the null condition, in the regularity regime of the classical local existence theory, assuming smallness of the angular derivatives. Unlike previous methods, which rely on decay in time, our approach is based solely on spatial decay. This provides a new perspective on the problem and offers potential applications to the study of Einstein's equations, which we intend to explore in future work.
\end{abstract}

\maketitle

\section{Introduction}
In seminal works, Klainerman \cite{klainerman_Longtime_1982, MR837683} and Christodoulou \cite{christodoulou_Global_1986} proved the global well-posedness of quasilinear wave equations 
in three space dimensions under the so-called ``null condition''. This gave rise to an extensive body of work, culminating in 
the proof of the global stability of Minkowski space-time \cite{christodoulou_Global_1993,lindblad_Global_2003, lindblad_The_2010,shen_Global_2023} (and, more recently, in the proof of the stability of Kerr), as well as in the global existence
 of nonlinear elastodynamics under the null condition \cite{agemi_Global_2000,sideris_Null_1996,sideris_Nonresonance_2000}.

The classical proofs of these results are based on two different methods: Klainerman's vector field method and Christodoulou's conformal mapping method. Over the years, several other approaches to the global existence problem have been developed, including the $r^p$-weight method of Dafermos and Rodnianski \cite{dafermos_New_2010a} and the space-time resonance method introduced by Germain, Masmoudi and Shatah \cite{germain_Global_2009,germain_Global_2012}. The goal of this paper is to give an alternative proof of this classical result by using the method of bilinear estimates recently developed by the second author and his collaborators \cite{2024Physical,Tu_2024,wang_2022,wang_2023}. This provides a new perspective on the global well-posedness of quasilinear wave equations: it imposes weaker restrictions on the initial data in terms of both regularity and decay, even when compared with the $r^p$-weight method. On the other hand, our method is related in some sense to Klainerman's vector field method, since we make use of angular derivatives of fractional order.

We consider the quasilinear hyperbolic equation of the form
\begin{align}
  \begin{cases} \Box u = h^{i \alpha \beta }\pa_i\pa_{\alpha}u\cdot\pa_{\beta} u, & \text{in}\  [0,T]\times \mathbb{R}^3, \\(u,\pa_{t} u)|_{t=0} = (u_0,u_1) , & \text{in}\  \mathbb{R}^3, \end{cases}	\label{quasi}
   \end{align} 

here and in what follows, $\Box = \pa^2_t-\Delta$ denotes the usual d'Alembertian, the lowercase Latin indices $i,j,k$ range from 1 to 3, and the Greek indices $\alpha,\beta,\gamma$ range from 0 to 3 in expressions, where the index $0$ stands for the time variable.

We assume that the coefficients  $h^{\alpha \beta \gamma}$ are constants satisfying the null condition, namely 
\begin{equation}
  h^{\alpha\beta \gamma}\xi_\alpha\xi_\beta\xi_\gamma = 0,\ \text{for any null vector}\ \xi.
\end{equation}

By a simple symmetrization procedure, the coefficients can further be assumed to satisfy 
\begin{equation}
  h^{\alpha \beta \gamma} = h^{\beta \alpha \gamma},\ h^{00\gamma}=0.
\end{equation} 

The critical regularity of \eqref{quasi} is governed by its scaling. Standard dimensional analysis shows that the critical Sobolev exponent for the norm $\|u \|_{H^s}$ is $s_c = \frac{5}{2}$. Indeed, the equation is invariant under the scaling \(u_\lambda(t,x) = \lambda^{-1} u(\lambda t, \lambda x)\), and a direct computation gives  $\| u_\lambda\|_{\dot{H}^s_x} = \lambda^{s-\frac{5}{2}}\|u\|_{\dot{H}^s_x} $, so that the homogeneous Sobolev norm is scale invariant precisely when $s = \frac{5}{2}$ in three dimensions.

In the classical local existence theory, the standard energy estimates
\begin{equation}
  \|(u,\pa_t u)(t)\|_{H^s\times H^{s-1}}\lesssim \exp\Big(\int^t_0 \|\pa^2 u \|_{L^\infty}\Big)  \|(u_0,u_1)\|_{H^s \times H^{s-1}} 
\end{equation}
together with the Sobolev embedding $H^{\frac{3}{2}+}_x\subseteq L^\infty_x$ in dimension 3, require initial data with regularity \(s > s_c + 1 = \frac{7}{2}\). In contrast, for semilinear wave equations, global solutions can be constructed under much lower regularity, sometimes even at the critical level, by exploiting the null structure and Strichartz estimates. However, for quasilinear equations, the loss of derivatives in the nonlinearity imposes a more stringent regularity requirement.

Our main result shows that global well-posedness can be achieved within this classical local existence regime, i.e., for \(s > s_c +1 =\frac{7}{2}\), without relying on any decay of the solution in time. Instead, we rely solely on spatial decay and angular regularity. Our assumptions on the initial data are almost scaling invariant: roughly speaking, the smallness condition reads
$$\|\pa u[0]\|_{H^{\frac{5}{2}+3\delta}}^{\frac{1}{2}}\|\Lambda_\omega^{2+\delta}\pa u[0]\|_{H^{\frac{1}{2}+\delta}}^{\frac{1}{2}}\ll 1,$$
under which we prove that the solution exists globally and satisfies uniform energy bounds. This is made possible by a combination of the angular Littlewood-Paley decomposition, novel bilinear estimates, and a new type of div-curl lemma developed in \cite{2024Physical}. In particular, we do not require the solution to decay in time, which is a crucial difference from all previous methods.

\begin{theorem}\label{main}
  Let $\delta>0$ be a sufficiently small real number, $\Lambda_\omega:= (1-\Delta_{\mathbb{S}^2})^{\frac{1}{2}}$. Given the initial data $\pa u[0]:=(\pa_x u_0,u_1)$ satisfying
  \begin{equation}
    \|\pa u[0]\|_{H^{\frac{5}{2}+3\delta}}\le  M,\ \| \Lambda_\omega^{2+\delta}\pa u[0]  \|_{H^{\frac{1}{2}+\delta}} \le N,
  \end{equation}
if the following almost scaling invariant quantity is sufficiently small, namely
\begin{equation}
  M^{1+\frac{\delta}{2}}N^{1+\frac{\delta}{2}} + M^{1-\frac{\delta}{4}}N^{1+\frac{ 5\delta}{4}} \ll 1,
\end{equation}
  then the Cauchy problem \eqref{quasi} has a unique global solution $u$ such that 
\begin{equation}
  \| \pa u\|_{C(H^{\frac{5}{2}+3\delta})}\lesssim M,\  \| \Lambda_\omega^{2+\delta}\pa u\|_{L^\infty(H^{\frac{1}{2}+\delta})}\lesssim N.
\end{equation}
Moreover, the solution depends continuously on the initial data in $C(H^{\frac{3}{2}+3\delta})$. More precisely, given two different solutions $u,v$ as above, we have 
\begin{equation}\label{lip}
  \|\pa u - \pa v\|_{C(H^{\frac{3}{2}+3\delta})}\lesssim \|\pa u[0]-\pa v[0] \|_{H^{\frac{3}{2}+3\delta}}.
\end{equation}
\end{theorem}

\begin{remark}
The almost scaling invariant nature of the smallness condition makes Theorem \ref{main} applicable, in a sense, to short pulse initial data, introduced by Christodoulou \cite{christodoulou_Formation_2009} and subsequently generalized in \cite{klainerman_On_2012}; more recently such data have been used in the study of nonlinear wave equations \cite{ding_On_2024,shen_Critical_2025}. In these works the initial data are concentrated in a thin shell of width $\eta$, are smooth in the angular variable $\omega\in\mathbb{S}^2$, and the smallness is imposed only on the angular directions and along the outgoing derivative $D_+:=\pa_t+\pa_r$, while the incoming derivative $D_-:=\pa_t-\pa_r$ is allowed to be large. Since our data are prescribed at $t=0$ rather than on a null hypersurface, we write such data, following the notation of \cite{ding_On_2024}, in the form
\begin{equation}
  u_0(x)=\eta^{2-\sigma}\varphi_0(\eta^{-1}\bar u,\omega),\qquad u_1(x)=\eta^{1-\sigma}\varphi_1(\eta^{-1}\bar u,\omega),
\end{equation}
where $\bar u:=r-t $ is an outgoing null coordinate (so that $\bar u(0,x)=r $), $\varphi_0,\varphi_1\in C_0^\infty((0,\infty)\times\mathbb{S}^2)$, $0<\sigma<1$ is fixed, and $\eta>0$ is sufficiently small; the parameters denoted by $\delta$ and $\varepsilon_0$ in \cite{ding_On_2024} correspond to our $\eta$ and $\sigma$, respectively. A direct computation gives
$$\|\pa u[0]\|_{H^{\frac{5}{2}+3\delta}}\sim\eta^{-\sigma-3\delta},\qquad \|\Lambda_\omega^{2+\delta}\pa u[0]\|_{H^{\frac{1}{2}+\delta}}\sim\eta^{2-\sigma-\delta},$$
so that the almost scaling invariant quantity in Theorem \ref{main} is of size $\eta^{1-\sigma-2\delta}\ll 1$ for $\eta$ sufficiently small and $\sigma \le 1-3\delta$. Hence the data above are admissible in Theorem \ref{main}.
\end{remark}

The paper is organized as follows. In Section 2, we introduce notation and preliminaries, including the angular Littlewood–Paley theory and commutator estimates. Section 3 is devoted to balance laws and local energy estimates with translations. In Section 4, we present a new type of div-curl lemma and derive the core null form estimates. Section 5 sets up the frequency envelope bootstrap argument. Finally, Section 6 completes the proof of global well-posedness and establishes the continuous dependence on the initial data.

\section{Preliminaries}
\subsection{Notations and conventions}
\begin{enumerate}
   \item We use $A\lesssim B $ to denote the statement that $A \le  CB$ for some absolute constant $C$, and $A \sim    B$ to denote the statement $A\lesssim B\lesssim A$. We shall choose small exponents such that 
   $$0<\delta_0\ll \delta \ll 1.$$
   \item Given a function $f(x)$ on $\mathbb{R}^{k}$,  we use $\mathcal{F}f(\xi) = \hat{f}(\xi)$ to denote its spatial Fourier transform, and $\mathcal{F}^{-1}g(x) = \check{g}(x)$ to denote its inverse Fourier transform.
   \item We express the points in $\mathbb{R}^3$ as 
$$x = (x_1,x_2,x_3) = (x^1,x^2,x^3) = r\omega,$$
where $\omega=(\omega_1,\omega_2,\omega_3)=\frac{x}{r}\in\mathbb{S}^2$ is the unit radial vector, and we use the convention
 $u^{y_0}(x) = u(x-y_0)$ for the translation of $u$ by $y_0$. In particular, we write $r_{y_0} := |x-y_0|$, and we use $\langle r_{y_0}\rangle$ to denote
 $$\langle r_{y_0}\rangle:=  (1+r_{y_0}^2)^{\frac{1}{2}} . $$

   \item For any scalar function $f$, we write the spherical integral in short as
$$\int_{\mathbb{S}^2}f=\int_{\mathbb{S}^2}f \mathrm{d}\omega := \int_{\mathbb{S}^2}f(r\omega)\mathrm{d}\mu_{\mathbb{S}^2}(\omega),\ \|f\|_{L^p_{\omega}}:=\Big(\ins |f(r\omega)|^p \mathrm{d}\mu_{\mathbb{S}^2}(\omega)\Big)^{\frac{1}{p}},$$
where $\mu_{\mathbb{S}^2}$ is the standard sphere measure. We also introduce the notation for the mixed Lebesgue norm $\|\cdot\|_{L^p_r L^q_\omega}$ as 
$$\ \|f\|_{L^p_{r}L^q_{\omega}}:= \Big(\Big(\ins |f(r\omega)|^{q}\mathrm{d}\mu_{\mathbb{S}^2}(\omega)\Big)^{\frac{p}{q}}\mathrm{d}r\Big)^{\frac{1}{p}},$$
and it is clear that $\|rf\|_{L^2_rL^2_\omega} = \|f\|_{L^2_x}$.
   \item We use $\pa_x$ to denote the spatial gradient, $\pa_\omega$ to denote the angular gradient on the sphere,  and $\pa$ to denote the full space-time gradient. The spatial gradient can be further decomposed into radial and angular components
$$\pa_i =\omega_i \pa_r+\slashed{\pa}_i := \omega_i \pa_r + r^{-1}\omega^j\Omega_{ji},$$
where $\Omega_{ij} = x_i\pa_j-x_j\pa_i$. It is clear that
$$|\slashed{\pa} u|^2:=\sum_{i}|\slashed{\pa}_i u|^2 = r^{-2}|\pa_\omega u|^2.$$
   \item For scalar-, vector- or tensor-valued quantities $A$ and $B$, we write $A\cdot B$ to denote a linear combination of componentwise products, with coefficients that are constants or functions uniformly bounded by an absolute constant $C$; e.g., $\pa \Pi\cdot\pa \Phi$ stands for a linear combination of $\pa_\alpha \Pi_\gamma \cdot \pa_\beta \Phi_{\delta\eta}$, and likewise $A\cdot B\cdot C$ for products of components of $A,B,C$. The dot product is distributive, $A\cdot(B+C)=A\cdot B+A\cdot C$, and translations act on the whole expression, $(A\cdot B)^{y_0}=A^{y_0}\cdot B^{y_0}$.
\end{enumerate}

\subsection{Frequency decomposition and Sobolev spaces on the sphere}
\subsubsection{Standard Littlewood-paley operators}
For a dyadic number $\lambda \in 2^{\mathbb{N}}, $ we define the standard inhomogeneous Littlewood-Paley operator $P_\lambda$. Let $\chi$ be the smooth cutoff to the region $[-1,1]^n$, and $P_\lambda,P_{\le\lambda}$ are defined by
\begin{equation}
  \widehat{P_{\le \lambda}f}(\xi) := \chi(\lambda^{-1}\xi)\hat{f}(\xi),\ P_{1}:= P_{\le 1}\ \text{and}\ P_{\lambda} = P_{\le\lambda} - P_{\le\frac{\lambda}{2}}\ \text{for}\  \lambda\ge 2.
\end{equation}
Thus, we have
\begin{equation}
  1 = \sum_{\lambda \in 2^{\mathbb{N}}} P_{\lambda}.
\end{equation}
We often use the shorthand
\begin{equation}
  u_{\lambda}:=P_\lambda u 
\end{equation}
for the Littlewood-Paley pieces of $u$.

Given dyadic numbers $\lambda,\lambda_1,\lambda_2$, we write $\lambda_{\text{max}}\ge \lambda_{\text{med}}\ge  \lambda_{\text{min}}$ for the maximum, median, and minimum of $\lambda,\lambda_1,\lambda_2$ respectively, and analogously for the angular frequencies $\nu,\nu_1,\nu_2$.  It is clear that $P_\lambda(u_{\lambda_1}v_{\lambda_2})$ vanishes unless $\lambda_{\text{max}}\sim \lambda_{\text{med}}$.

It is notable that, by the boundedness of the kernel, for any Lebesgue norm $\|\cdot\|_{X}$ we have 
\begin{equation}\label{re}
  \|u_\lambda \pa_x v_\mu\|_{X}\lesssim \sup_{y_0}\mu\|u_\lambda v_\mu^{y_0}\|_{X}.
\end{equation} 
\subsubsection{Angular Littlewood-Paley operators}
Following Guo \cite{guo_Global_2023}, we introduce angular Littlewood-Paley type operators via spectral decomposition of the Laplacian on ${\mathbb {S}}^2.$ Let $\Pi_n$ denote the $L^2$-projector onto the $n$-th eigenspace of the
spherical Laplacian $-\Delta_{\mathbb{S}^2}$ associated to the eigenvalue $\lambda_n = n(n+1)$. The following fact about angular regularity holds (see \cite[Section 2.8.4]{atkinson_Spherical_2012})  
\begin{equation}
  \Pi_n f (x) = \int _{{\mathbb {S}}^2}f(r \omega' ){\mathfrak {Z}}_n(\langle \omega' ,\omega\rangle )d\mu _{{\mathbb {S}}^2}(\omega' ),
\end{equation}
where
\begin{equation}
   {\mathfrak {Z}}_n(x)=\frac{2n+1}{4\pi }L_n(x),\  L_n(z)=\frac{1}{2^n n! }\frac{\mathrm{d}^n}{\mathrm{d}z^n}[(z^2-1)^n]. 
\end{equation}

 As before, for $\nu\in 2^{\mathbb {N}}$ we define the angular Littlewood-Paley operators $R_\nu, R_{\le  \nu}$ by

\begin{equation}
  R_{\le \nu }f  =\sum _{n\ge 0}\chi(\nu^{-1}n)\Pi_n f ,\ R_{1}:= R_{\le 1}\ \text{and}\ R_{\nu} = R_{\le\nu} - R_{\le\frac{\nu}{2}}\ \text{for}\ \nu\ge 2.
\end{equation}

These operators are bounded on $L^2$ and self-adjoint, and their key properties mirror those of the standard Littlewood-Paley operators. They are summarized in the following proposition.

\begin{proposition}
\begin{enumerate}
  \item $R_\nu$ commutes with the following operators 
  \begin{equation}
    [\pa_\omega,R_\nu]=[\pa_r ,R_\nu]=[P_{\lambda},R_\nu]=0. 
  \end{equation} 
  \item $R_\nu$ has the almost orthogonal properties in the sense that 
 \begin{equation}
   1=\sum _{\nu}R_\nu,\  R_\nu R_{\nu'}=0\ {unless}\  \nu\sim \nu'  ,\ \text{therefore}\  \|f\|_{L^2_\omega}^2 \sim  \sum_\nu \|R_\nu f\|_{L^2_\omega}^2.
 \end{equation}
Moreover, for $\{i,j,k\} = \{1,2,3\} $, we have
\begin{equation}
  R_{\nu_i}(R_{\nu_j}f\cdot R_{\nu_k}g) =0,\ \text{unless}\  \nu_{\text{max}}\sim \nu_{\text{med}}.
\end{equation}
  \item %The kernels for $R_\nu$ are integrable on the sphere,i.e.
  %$$R_\nu f(r \omega) = \int_{\mathbb{S}^2}f(r\omega')K_\nu(\omega,\omega')\mathrm{d}\mu_{\mathbb{S}^2}(\omega'),\  \sup_{\omega}\|K_\nu(\omega,\omega')\|_{L^1_{\omega'}}+\sup_{\omega'}\|K_\nu(\omega,\omega')\|_{L^1_{\omega}}\lesssim 1,$$
   $R_\nu, R_{\le \nu }$ are bounded on $L^q, 1\le q\le \infty$, namely \begin{equation}
      \Vert R_\nu f\Vert _{L^q}+\Vert R_{\le \nu }f\Vert _{L^q}\lesssim \Vert f\Vert _{L^q}. 
   \end{equation}
  \item We have Bernstein inequalities as
\begin{equation}
 \Vert \pa_\omega R_\nu f\Vert _{L^q}\sim \nu \Vert R_\nu f\Vert _{L^q} .  
\end{equation}
\end{enumerate}
\end{proposition}
We refer the reader to \cite[Appendix A.1]{guo_Global_2023} for the proof of this proposition.  

Although $R_\nu$ does not commute with $\pa_{x}$, we observe that the $\omega_i$ are polynomials of degree 1 on the sphere, which allows us to write
$$R_\nu(\pa_i f) = R_\nu[(\omega_i \pa_r + r^{-1}\omega^j\Omega_{ji})  f] =\sum_{\nu'\sim \nu} R_\nu(\omega_i R_{\nu'}\pa_rf + r^{-1}\omega^j R_{\nu'}\Omega_{ji}f )  = \sum_{\nu'\sim  \nu}R_\nu(\pa_x R_\nu' f).  $$

We also use the shorthand for the angular Littlewood-Paley pieces of $u$, namely 
\begin{equation}
  u_{\lambda,\nu}:=P_\lambda R_\nu u .
\end{equation}

\subsubsection{Sobolev spaces on the sphere}
We define the operator $\Lambda_\omega^s = (1-\Delta_{\mathbb{S}^2} )^{\frac{s}{2}}$ by 
\begin{equation}
  \Lambda_\omega^s f := \sum_n (\lambda_n + 1)^{\frac{s}{2}}\Pi_n f, 
\end{equation}
and we let the spherical Sobolev spaces $H^{s}(\mathbb{S}^2)$ be the Banach spaces equipped with the norm
\begin{equation}
  \left\| f \right\|_{H^{s}(\mathbb{S}^2)} := \left\| \Lambda^{s}_\omega f \right\|_{L^2_\omega}.
\end{equation}  

It is clear that 
\begin{equation}
  \left\| f \right\|^2_{H^{s}(\mathbb{S}^2)} \sim  \sum_{\nu} \nu^s \left\| R_\nu f \right\|_{L^2_\omega}^2.
\end{equation} 

Notice that 
$$\begin{aligned}
 & \|R_\nu f(r\omega)\|_{L^\infty_\omega}
 \lesssim&  \left\| f \right\|_{L^2_\omega} \Big\| \sum_{n}\varphi(\nu^{-1}n) {\mathfrak {Z}}_n(\langle \omega',\omega\rangle ) \Big\|_{L^2_{\omega'}}
 \lesssim& \nu  \left\| f \right\|_{L^2_\omega}. \\
\end{aligned}$$
After summation over $\nu$, we obtain the Sobolev embedding on the sphere
\begin{equation}
  \left\| f \right\|_{L^\infty_\omega}\lesssim \left\| f \right\|_{H^{1+\delta_0}(\mathbb{S}^2 )} = \| \Lambda^{1+\delta_0}_\omega f\|_{L^2_\omega}.
\end{equation}
\subsection{Expressions for null forms}

It is well known that the null forms can be decomposed into linear combinations of certain types of null forms. More precisely, we have the following lemma.
\begin{lemma}
  If the affine quadratic form $h^{i \alpha\beta }\pa_{i}\pa_{\alpha } v_1 \pa_{\beta}v_2  $ satisfies the null condition, then 
  \begin{align}
    &~ h^{i\alpha\beta }\pa_i\pa_{\alpha}v_1\pa_{\beta} v_2 = N_{\alpha \beta}(\pa_{x} v_1,v_2) + N_0(\pa_{x} v_1,v_2),
   \end{align}
  where
\begin{equation}
    N_{\alpha \beta}(f,g) = \pa_\alpha f  \cdot \pa_\beta g-  \pa_\alpha g \cdot \pa_\beta f, N_0(v_1,v_2) = \pa_t v_1 \cdot \pa_t v_2 - \pa_x v_1 \cdot \pa_x v_2. 
\end{equation}
Moreover, we have 
\begin{equation}
h^{i\alpha\beta }  \pa_i v_2 \cdot \pa_\alpha v_2 \cdot \pa_\beta v_1 =\pa v_1 \cdot N_{\alpha \beta}(v_1,v_2) + \pa v_1 \cdot N_0( v_1,v_2).
\end{equation}
\end{lemma}  
Its proof can be found in \cite[Lemma 12.2.1]{li_Nonlinear_2017}.

Notice that, in terms of the operators $D_\pm$ and the angular derivatives, the above null forms can be expanded as follows:
$$  \begin{aligned}
 &~ N_{0i}(f,g) = \frac{1}{2}\omega_i(D_- f\cdot D_+ g - D_- g\cdot D_+ f) + r^{-1}(\pa_0 f\, \omega^j \Omega_{ji}g - \pa_0 g\,  \omega^j \Omega_{ji}f), \\
 &~ N_{ij}(f,g) =r^{-1} (\pa_j g\, \omega^k \Omega_{ki}f - \pa_i g\,  \omega^k \Omega_{ki}f + \pa_r f\, \Omega_{ij}g) , \\
 &~ N_0(f,g) = \frac{1}{2}(D_+ f\cdot D_- g + D_+ g\cdot D_- f) - r^{-1}( \pa^i g\, \omega^k\Omega_{ki}f - \pa_r f\,  \omega^k \omega^i \Omega_{ki}g),
\end{aligned}$$
where we use the shorthand $D_{\pm} = \pa_t \pm \pa_r$. Thus we can further expand the null form as 
\begin{equation}
  h^{i \alpha\beta }\pa_{i}\pa_{\alpha } v_1 \pa_{\beta}v_2  = D_{\pm }\pa_{x} v_1 \cdot D_{\mp }v_2 + \slashed{\pa} \pa_{x} v_1 \cdot \pa v_2 +  \pa_{x} \pa v_1 \cdot \slashed{\pa} v_2.
\end{equation}

In what follows, we always use the above expansion in the region $\{r\ge 1\}$. We introduce the notation  
\begin{align}
    &~ G(v_1, v_2)=D_{\pm }v_1 \cdot D_{\mp }v_2+ \pa v_1 \cdot \slashed{\pa}v_2+\slashed{\pa} v_1 \cdot \pa v_2,
\end{align}
and it is clear that
\begin{align}
  &  h^{i \alpha\beta }\pa_{i}\pa_{\alpha } v_1 \cdot \pa_{\beta}v_2 = G(\pa_x v_1,v_2),\
  &  h^{i\alpha\beta }  \pa_i v_2 \cdot \pa_\alpha v_2 \cdot \pa_\beta v_1 = \pa v_1 \cdot G(v_1,v_2).\label{r1}
\end{align}

\subsection{Multilinear expressions}
Following Tao \cite{tao_Global_2001}, we introduce a convenient notation for describing multi-linear expressions of product type.

For multilinear operators, let $L$ be the integral form
\begin{equation}
  L(u_1,u_2, \cdots ,u_k)(x)  = \int K(y_1, \cdots ,y_k)u_1^{y_1}(x)\cdots u_k^{y_k}(x)\mathrm{d}y,
\end{equation} 
where $K$ can be an integrable kernel or, more generally, a bounded measure (including, for example, product type expressions). The kernel may change from line to line, but we always require these kernels to have uniformly bounded mass.

This $L$ notation will turn out to be useful for expressing matrix coefficients, Littlewood-Paley multipliers $P_\lambda$, commutator expressions, etc., whenever these structures are not being exploited. For example, the $L$ notation is invariant under inserting or removing the standard Littlewood-Paley operators
\begin{align}
 &~ L(P_\lambda u_1,u_2, \cdots ,u_k) =L( u_1,u_2, \cdots ,u_k),\\
 &~ P_\lambda L( u_1,u_2, \cdots ,u_k) =L( u_1,u_2, \cdots ,u_k).
\end{align} 
The same holds for $P_{>\lambda},\ P_{<\lambda}$, etc. Furthermore, this notation also interacts well with derivatives of Littlewood-Paley operators
\begin{align}
 &~ L(\pa_x P_\lambda u_1,u_2, \cdots ,u_k) =\lambda L( u_1,u_2, \cdots ,u_k).
\end{align}
And the reverse equality is also true for $\lambda\gg 1$
\begin{equation}
  L( P_{\lambda}u_1,u_2, \cdots ,u_k) = \lambda^{-1}L(\pa_x P_\lambda u_1,u_2, \cdots ,u_k). 
\end{equation}
Multilinear estimates are not invariant under separate translations of the factors. To obtain similar bounds for the $L$ notation, we shall allow translations in the multilinear estimates. For example, if we have the bounds for 
$$\sup_{y_0,y_1}\| u_1 u_2^{y_0} u_3^{y_1} \|_{L^p}\le C,$$
then for any multilinear form $L$ with integrable kernel we have 
\begin{align*}
  &~ \|L(u_1,u_2,u_3)\|_{L^p} \\
  \le&~ \sup_{\left\| g \right\|_{L^{p'}}\le 1} \int K(y_1,y_2,y_3)u_1^{y_1}(x)u_2^{y_2}(x)u_3^{y_3}(x)g(x)\mathrm{d}y\mathrm{d}x\\
  =&~ \sup_{\left\| g \right\|_{L^{p'}}\le 1} \int K(y_1,y_2,y_3)u_1(x)u_2^{y_2-y_1}(x)u_3^{y_3-y_1}(x)g(x)\mathrm{d}y\mathrm{d}x\\
  \le&~ \sup_{y_0,y_1}\| u_1 u_2^{y_0} u_3^{y_1} \|_{L^p} \left\| K \right\|_{L^1}
  \lesssim   C.
\end{align*}
\subsection{Commutator estimates  }
\begin{lemma}\label{pcom}
  Given Schwartz functions $f,g$, let $f_\mu:=P_\mu f$, then we have 
  \begin{equation}\label{comm}
    [P_\lambda, f_\mu]g =\lambda^{-1}\mu L(f_\mu, g),
  \end{equation}
  where $[P_\lambda, f_\mu]g :=P_\lambda(f_\mu g)-f_\mu P_\lambda g. $
\end{lemma}
\begin{proof}
  We compute directly:
  \begin{align*}
    &~ [P_\lambda,f_\mu]g(x) = \lambda^n \int_{\mathbb{R}^n}\check{\chi}(\lambda y)[f_\mu(x-y)-f_\mu(x)]g(x-y)\mathrm{d}y\\
    =&  -\lambda^{n}\int_{0}^1
  \int_{\mathbb{R}^n} y\check{\chi}(\lambda y)\cdot \pa_x f_\mu(x-sy)g(x-y)\mathrm{d}s\mathrm{d}y \\
=&~ -\lambda^{-1}\int_{0}^1
  \int_{\mathbb{R}^n} \lambda^{n+1}y\check{\chi}(\lambda y)\cdot \pa_x f_\mu(x-sy)g(x-y)\mathrm{d}s\mathrm{d}y\\
=&~ \lambda^{-1}L(\pa_x f_\mu, g) = \lambda^{-1}\mu L( f_\mu, g). \end{align*}
The last line follows from the fact that the kernels $\lambda^{n+1}y\check{\chi}(\lambda y)$ are integrable, with uniformly bounded $L^1$ norm. 
\end{proof}

We also have the following commutator estimate for the angular operators $R_\nu$, whose proof can be found in the Appendix.
\begin{lemma}\label{ancom}
  Given Schwartz functions $f,g$ on the sphere $\mathbb{S}^2$, we have 
  \begin{equation}\label{angular}
   \| [R_\nu,f]g \|_{L^2_\omega}\lesssim \nu^{-1}\|\pa_\omega f  \|_{L^\infty_\omega}\|  g  \|_{L^2_\omega} .
  \end{equation}
\end{lemma}
Using the above lemma, we have 
\begin{proposition}\label{2.5}
 \begin{equation}
   \|\pa_x R_\nu f\|_{L^2_{\omega}}\lesssim \|\pa_x f\|_{L^2_\omega}.
 \end{equation}
\end{proposition}
\begin{proof}
  It is clear that
  \begin{align*}
    &~ \|\pa_{x_i} R_\nu f\|_{L^2_{\omega}} \lesssim \|[\pa_{x_i}, R_\nu] f\|_{L^2_{\omega}} + \|R_\nu (\pa_{x_i} f)\|_{L^2_{\omega}}\\
    \lesssim&~ \|[R_\nu,\omega_i \pa_r + r^{-1}\omega^j\Omega_{ji}] f\|_{L^2_{\omega}} + \|\pa_{x_i} f\|_{L^2_{\omega}}\\
    \lesssim&~ (\nu^{-1}\|\pa_\omega \omega \|_{L^\infty_\omega} +1)\|\pa_x f\|_{L^2_\omega}\\
    \lesssim&~ \|\pa_x f\|_{L^2_\omega},
   \end{align*}
where we used Lemma \ref{ancom} and the fact that the rotation generators $\Omega_{ji}$ commute with $R_\nu$.
\end{proof}
\subsection{Hardy type inequalities}
The proof of the following Hardy inequality can be found in \cite[Theorem 2.57]{bahouri_Fourier_2011}.
\begin{lemma} For $0\le s<\frac{3}{2}$,
 we have
 \begin{equation}
   \int_{\mathbb{R}^3}\frac{|f(x)|^2}{|x|^{2s}}\mathrm{d}x\lesssim\| f \|_{\dot{H}^s}^2.
 \end{equation}
\end{lemma}

With the help of the Hardy inequalities, we have the following trace type inequalities
\begin{lemma}
  For $0\le\varepsilon\le 1$, we have
  \begin{equation}\label{2.6}
    \|r^{1-\varepsilon} f_\lambda\|_{L^\infty_r L^2_\omega}\lesssim \lambda^{\frac{1}{2}+\varepsilon} \|f_\lambda\|_{L^2},\ \|r^{1-\varepsilon} f\|_{L^\infty_r L^2_\omega}\lesssim \|f\|_{H^{\frac{1}{2}+\delta_0+\varepsilon}}.
  \end{equation}
\end{lemma}
\begin{proof}
  We first estimate the frequency pieces $f_\lambda$: 
\begin{align*}
  &~ r^{2-2\varepsilon} \ins f_\lambda^2 \\
  =& -r^{2-2\varepsilon}\int_{r}^{+\infty}\ins\pa_\rho (f_\lambda^2)\mathrm{d}\rho \mathrm{d}\omega\\
   \le&~ \Big(\frac{r}{\rho}\Big)^{2-2\varepsilon}\int_{r}^{+\infty}\ins \rho^{-2\varepsilon}|f_\lambda||\pa_\rho f_\lambda|\ \rho^2\mathrm{d}\rho \mathrm{d}\omega\\
  \lesssim&~ \|r^{-2\varepsilon}f_\lambda\|_{L^2} \|\pa_r f_\lambda\|_{L^2}
  \lesssim  \lambda^{1+2\varepsilon}\|f_\lambda\|^2_{L^2},
\end{align*}
where we used the Hardy inequality and the Bernstein inequality. After summation over the frequencies $\lambda$, we obtain \eqref{2.6}.
\end{proof}

\section{Balance laws and local energy estimates}

In this section we derive local energy estimates (also known as Morawetz or KSS estimates), which will serve as one of the main tools in our proof. Due to the presence of translations, we require a translated version of the local energy estimates for our solution $v$. Specifically, we aim to establish:
$$\sup_{y_0}\int_{\mathbb{R}\times \mathbb{R}^3} r^{-(1-\delta_0)}\langle r \rangle^{-2\delta_0}[(\pa v^{y_0})^2 + r^{-2} (v^{y_0})^2] \mathrm{d}t \mathrm{d}x\lesssim \int_{\mathbb{R}^3} (\pa v(0))^2 \mathrm{d}x.$$
This estimate holds for free waves, since they have no source terms and the $\dot{H}^1$ norm on the right-hand side is translation invariant. In our case, however, we set $v = P u $, where $P$ is a Littlewood-Paley operator. To derive the estimate in this setting, we must control the source term $N(v)$ with translations, which essentially reduces to controlling
\begin{equation*}
  \| (\pa v^{y_0} + r^{-1}v^{y_0})\cdot [N(v)]^{y_0} \|_{L^1_{t,x}} = \| (\pa v + r^{-1}_{y_1}v)\cdot N(v) \|_{L^1_{t,x}},
\end{equation*}
where $y_1 = -y_0$. We will see that such $L^1_{t,x}$ estimates can be derived by combining the translated local energy estimates with the bilinear $L_{t,x}^2$ estimates provided by the div-curl lemma in Section 4.

\subsection{Balance laws}

We shall study the equation
\begin{equation}\label{model}
  \Box_{\bar{g}}v:=\Box v + \bar{h}^{\alpha\beta}\pa_{\alpha}\pa_{\beta}v = N(v),
\end{equation}
where $v$ is a Schwartz solution and $N(v)$ is the corresponding source term. The coefficients $\bar{h}^{\alpha \beta} = -h^{\alpha\beta\gamma}\pa_\gamma u$ are assumed
to be sufficiently small in $L^\infty$; we will verify this assumption under the bootstrap hypotheses in Section 6. Here and below, $\bar{g}=\bar{h}+I$ denotes the metric, and the subscript in $\Box_{\bar{g}}$ always refers to the metric $\bar{g}$.

We will establish the integrated balance laws for $v^{y_0}$. For simplicity, we define the energy and momentum densities as
$$
\begin{aligned}
  &~ e(v) := \frac{1}{2}(|\pa_t v|^2+ |\pa_r v|^2 + r^{-2}|\pa_\omega v|^2  ) ,\ e_{\pm}(v) := \frac{1}{4} (D_\pm v)^2,\\
  &~ e_\omega(v) := \frac{1}{2r^2}|\pa_\omega v|^2 = \frac{1}{2}|\slashed{\pa} v|^2,\ e_{t,r}(v) =  \frac{1}{2}[(\pa_t v)^2+(\pa_r v)^2],\\
  &~ p(v) := \pa_t v \cdot \pa_{r}v = e_{+}(v)-e_{-}(v).
\end{aligned}$$

We first apply a translation to \eqref{model}, namely
\begin{equation}\label{model-t}
  \Box_{\bar{g}^{y_0}}v^{y_0}= N(v)^{y_0},
\end{equation}

We first establish the energy balance law. Multiplying the translated equation \eqref{model-t} by $r^2\pa_t v^{y_0}$ and integrating over the sphere, we obtain 
\begin{equation}
    \begin{aligned}
  &~ \pa_t\ins \big( e(v^{y_0})+\frac{1}{2}(\bar{h}^{00}( \pa_0 v)^2 - \bar{h}^{ij}\pa_i v \pa_j v)^{y_0}\big) r^2\mathrm{d}\omega\\
  -&~ \pa_r\ins \big( p(v^{y_0})-\omega_i(\pa_t v\cdot \bar{h}^{i\beta} \pa_\beta v)^{y_0}\big) r^2\mathrm{d}\omega\\
 =&~ \ins \big( \pa_t v^{y_0}\cdot N(v)^{y_0} + (\pa_\alpha \bar{h}^{\alpha \beta}\pa_t v \pa_\beta v - \frac{1}{2}\pa_t \bar{h}^{\alpha \beta} \pa_\alpha v \pa_\beta v)^{y_0}\big) r^2\mathrm{d}\omega,
   \end{aligned}
\end{equation}

We next establish the momentum balance law. Multiplying by $r^2\pa_r v^{y_0} + rv^{y_0}$ and integrating over the sphere, we obtain 
\begin{equation}
  \begin{aligned}
  &~ \pa_t\ins \big( p(v^{y_0}) +r^{-1}(v\pa_t v)^{y_0} + (\pa_r v^{y_0} + r^{-1}v^{y_0})(\bar{h}^{0\beta}\pa_\beta v)^{y_0}\big) r^2\mathrm{d}\omega \\
  -&~ \pa_r\ins \big( e_{t,r}(v^{y_0}) +\frac{1}{2 r^2}\pa_r\big(r(v^{y_0})^2\big) - \pa_r v^{y_0}\,\omega_i (\bar{h}^{i\beta}\pa_\beta v)^{y_0} + \frac{1}{2}(\bar{h}^{\alpha \beta}\pa_\alpha v \pa_\beta v)^{y_0}\big) r^2\mathrm{d}\omega\\
  +&~ \frac{1}{r^2}\pa_r\ins e_\omega(v^{y_0}) r^4\mathrm{d}\omega + \ins  \pa_i(v\bar{h}^{i\beta}\pa_\beta v)^{y_0}r \mathrm{d}\omega - \ins  [\pa_\alpha,\pa_r]v^{y_0} \cdot (\bar{h}^{\alpha \beta}\pa_\beta v)^{y_0} r^2\mathrm{d}\omega  \\
  =&~ \ins \big( (\pa_r v^{y_0}+r^{-1}v^{y_0})\cdot( N(v)^{y_0} + (\pa_\alpha \bar{h}^{\alpha \beta} \pa_\beta v)^{y_0}) -\frac{\omega^i}{2}(\pa_i \bar{h}^{\alpha \beta}\pa_\alpha v \pa_\beta v)^{y_0}\big) r^2\mathrm{d}\omega,
   \end{aligned}
\end{equation}

Using the dot-product (schematic) notation introduced in Section 2.1, we can rewrite the above balance laws as follows 
\begin{equation}\label{energy}
  \begin{aligned}
  &~ \pa_t\ins \big( e(v^{y_0})+(\bar{h}\cdot \pa v\cdot \pa v)^{y_0}\big) r^2\mathrm{d}\omega\\
  & - \pa_r\ins \big( p(v^{y_0})+(\bar{h}\cdot \pa v\cdot \pa v)^{y_0}\big) r^2\mathrm{d}\omega\\
 =&~ \ins \big( \pa_t v^{y_0}\cdot N(v)^{y_0} + (\pa\bar{h}^{\alpha \beta}\cdot \pa_\alpha v\cdot \pa_\beta v)^{y_0}\big) r^2\mathrm{d}\omega,
   \end{aligned}
\end{equation}
and the momentum balance law
\begin{equation}\label{momentum}
  \begin{aligned}
  &~ \pa_t\ins \big( p(v^{y_0}) + (\bar{h}\cdot \pa v\cdot \pa v)^{y_0}+r^{-1}(1+\bar{h}^{y_0})(v\cdot \pa v)^{y_0}\big) r^2\mathrm{d}\omega \\
 & - \pa_r\ins \big(e_{t,r}(v^{y_0}) +\frac{1}{2 r^2}\pa_r\big(r(v^{y_0})^2\big) + (\bar{h}\cdot \pa v\cdot \pa v)^{y_0} \big) r^2\mathrm{d}\omega\\
  & +\frac{1}{r^2}\pa_r\ins e_\omega(v^{y_0}) r^4\mathrm{d}\omega + \ins \big[\pa_x(\bar{h}\cdot v\cdot \pa_\beta v) + \bar{h}\cdot \pa v \cdot \pa v \big]^{y_0}r\,\mathrm{d}\omega\\
  =&~ \ins \big( (\pa_r v^{y_0}+r^{-1}v^{y_0})\cdot( N(v)^{y_0} + (\pa_\alpha \bar{h}^{\alpha \beta} \pa_\beta v)^{y_0}) + (\pa\bar{h}^{\alpha \beta}\cdot \pa_\alpha v\cdot \pa_\beta v)^{y_0}\big) r^2\mathrm{d}\omega,
   \end{aligned}
\end{equation}

For simplicity, we define the source densities by
\begin{equation}
  \begin{aligned}
  &~ \mathbf{F}_e(v^{y_0})=  \pa_t v^{y_0}\cdot [N(v)]^{y_0} + \bar{Q}_0(v^{y_0}) ,\\
  &~ \mathbf{F}_m(v^{y_0})=(\pa_r v^{y_0}+ r^{-1}v^{y_0})\cdot [N(v)]^{y_0}+\bar{Q}_1(v^{y_0}).
   \end{aligned}
\end{equation}

And we define the translation invariant functional $E(v)$ by
$$E(v) =\int_{\mathbb{R}^3}e(v(0))\mathrm{d}x + \sup_{y_0}\left\| {\mathbf{F}}_e(v^{y_0}) \right\|_{L^1_{t,x}}+ \sup_{y_0}\left\| {\mathbf{F}}_m(v^{y_0}) \right\|_{L^1_{t,x}}. $$
By the smallness of the coefficients and the energy balance law, it is clear that for $T >0$ we have 
  \begin{align}
   &~ \left\| \pa v(T) \right\|^2_{L^2} \lesssim\int_{\mathbb{R}^3}\big(e(v^{y_0})+\frac{1}{2} \bar{h}^{00}| \pa_0 v|^2 - \bar{h}^{ij}\pa_i v \pa_j v\big)\mathrm{d}x \nonumber\\
   \lesssim&~ \int_{\mathbb{R}^3}e(v(0))\mathrm{d}x + \int_{[0,T]\times \mathbb{R}^3}{\mathbf{F}}_e(v)\mathrm{d}x\nonumber\\
   \lesssim&~ E(v ).
   \end{align}
\subsection{Local energy estimates}
In this subsection we establish the local energy estimate for \eqref{model-t} by means of the momentum balance law \eqref{momentum}.

\begin{proposition}\label{Le}
  Assume that coefficients $\bar{h}$ satisfy
 \begin{equation}\label{small_con}
   \|\langle r \rangle^{\delta_0}\bar{h}\|_{L^\infty} \le \varepsilon\ll 1. 
 \end{equation} 
   Let $v$ be a Schwartz solution to \eqref{model}, then for $\delta_0>0$, we have 
  \begin{equation}
  \begin{aligned}
    &~ \sup_{y_0} \|r^{-\frac{1-\delta_0}{2}}\langle r \rangle^{-\delta_0}\pa v^{y_0} \|^2_{L^2_{t,x}}+ \|r^{-\frac{3-\delta_0}{2}}\langle r \rangle^{-\delta_0} v^{y_0} \|^2_{L^2_{t,x}}\\
    =&~ \sup_{y_0}\int_{\mathbb{R}_+ \times \mathbb{R}^3}r^{-(1-\delta_0)}\langle r \rangle^{-2\delta_0} (e(v^{y_0})+ r^{-2}(v^{y_0})^2)\mathrm{d}t\mathrm{d}x \lesssim E(v).
   \end{aligned}
\end{equation}
\end{proposition}
\begin{proof}
 We use a standard bootstrap argument. We make the following bootstrap assumption:
\begin{equation}\label{ba_local}
  \sup_{y_0}\int_{\mathbb{R}_+ \times \{R\le r\le 2R\}} e(v^{y_0})+ R^{-2}(v^{y_0})^2 \mathrm{d}t\mathrm{d}x \le C \max\{R^{1-\frac{\delta_0}{2}},R\}  E(v),
\end{equation}
or equivalently
\begin{equation}
  \sup_{y_1}\int_{\mathbb{R}_+ \times \{R\le r_{y_1}\le 2R\}} e(v)+ R^{-2}v^2 \mathrm{d}t\mathrm{d}x \le C \max\{R^{1-\frac{\delta_0}{2}},R\}  E(v),
\end{equation}
where $C$ is a sufficiently large constant to be chosen. Once the constant is improved, a continuity argument yields the desired bounds.

Multiplying \eqref{momentum} by $\alpha(r) = \frac{r}{r+R}$ and integrating over $(t,r)\in  [0,T]\times [0,+\infty)$, we have
\begin{equation}\label{int}
  \begin{aligned}
    &~ \int_{[0,T]\times \mathbb{R}^3} \alpha'(r)(e_{t,r}(v^{y_0})+(\bar{h}\cdot \pa v\cdot \pa v)^{y_0})-\frac{\alpha''(r)}{2r}(v^{y_0})^2-\Big(\frac{\alpha(r)}{ r^2}\Big)' r^2 e_\omega(v^{y_0}) \ \mathrm{d} t \mathrm{d} x \\
    =&~ -\int_{\mathbb{R}^3}\alpha(r) [p(v^{y_0}) + (\bar{h}\cdot \pa v\cdot \pa v)^{y_0}+r^{-1}(1+\bar{h}^{y_0})(v\cdot \pa v)^{y_0}]\mathrm{d}x\Big|^T_0\\
    &~ \quad +\int_{[0,T]\times \mathbb{R}^3}\alpha(r)\mathbf{F}_m(v)^{y_0}\mathrm{d}t \mathrm{d}x+\int_{[0,T]\times \mathbb{R}^3} \frac{ (\bar{h}\cdot v\cdot \pa v)^{y_0}}{(r+R)^2} + \frac{(\bar{h}\cdot \pa v\cdot \pa v)^{y_0} }{r+R} \mathrm{d}t\mathrm{d}x,
   \end{aligned}
\end{equation}
where
\begin{equation}
  \alpha' = \frac{R}{(r+R)^2},\ -\alpha''(r) = \frac{2R}{(r+R)^3},\ -\Big(\frac{\alpha(r)}{r^2}\Big)' = \frac{2r+R}{r^2(r+R)^2}.
\end{equation}

Since the coefficients $\bar{h}$ are sufficiently small, the left-hand side of \eqref{int} is positive, and we have 
\begin{equation}\label{le1}
  LHS \ge \frac{1}{10} \int^{T}_0\int_{\frac{R}{2}\le r\le R}\frac{e(v^{y_0})}{R} + \frac{(v^{y_0})^2}{R^3}\mathrm{d}t\mathrm{d}x.
\end{equation}

The first and second terms on the right-hand side of \eqref{int} can be bounded using the Hardy inequalities, so it remains to control the last term. 
To this end, we decompose the integral as
\begin{align*}
  &~ \int_{[0,T]\times \mathbb{R}^3} \frac{ (\bar{h}\cdot v\cdot \pa v)^{y_0}}{(r+R)^2} +   \frac{(\bar{h}\cdot \pa v\cdot \pa v)^{y_0} }{r+R}\,\mathrm{d}t\mathrm{d}x\\
  =&~ \int_{[0,T]\times \mathbb{R}^3} \frac{\bar{h}\cdot v\cdot \pa v }{(r_{y_1}+R)^2} +  \frac{\bar{h}\cdot \pa v\cdot \pa v}{r_{y_1}+R} \,\mathrm{d}t\mathrm{d}x\\
 =&~ \sum_{R'\le 1\ \text{dyadic}}\int_{\{R'\le r_{y_1}\le 2R'\}}  \frac{\bar{h}\cdot v\cdot \pa v }{(r_{y_1}+R)^2} +  \frac{\bar{h}\cdot \pa v\cdot \pa v}{r_{y_1}+R} \,\mathrm{d}t\mathrm{d}x\\
  & + \int_{\{r_{y_1}\ge 1\}}  \frac{\bar{h}\cdot v\cdot \pa v }{(r_{y_1}+R)^2} +  \frac{\bar{h}\cdot \pa v\cdot \pa v}{r_{y_1}+R} \,\mathrm{d}t\mathrm{d}x
\end{align*}
where $y_1 = -y_0$. In the region $\{r_{y_1}\le 1\}$, we have 
\begin{align*}
  &~ \sum_{R'\le 1\ \text{dyadic}}\int_{R'\le r_{y_1}\le 2R'}  \frac{1}{(r_{y_1}+R)^2}\bar{h}\cdot v\cdot \pa v\\
  \lesssim&~   C \|\bar{h}\|_{L^\infty} E(v) \sum_{R'\le 1 \ \text{dyadic}} \frac{R'^{2-\frac{\delta_0}{2}} }{(R'+R)^2} +  \frac{R'^{1-\frac{\delta_0}{2}} }{ R'+R }\\ 
  \lesssim&~ C \varepsilon  \max\{R^{-\frac{\delta_0}{2}},1\} E(v), 
\end{align*}
 
and the last line follows by applying the bootstrap assumption \eqref{ba_local}: if $R<1$, then
\begin{align}
   &~\sum_{R'\le 1 \ \text{dyadic}} \frac{R'^{2-\frac{\delta_0}{2}} }{(R'+R)^2} + \frac{R'^{1-\frac{\delta_0}{2}} }{ R'+R }\nonumber\\
   \lesssim&~  \sum_{R'\le R \ \text{dyadic}} R^{-1} R'^{1-\frac{\delta_0}{2}} + \sum_{R'\ge R \ \text{dyadic}}  R'^{-\frac{\delta_0}{2}}
  \lesssim  R^{-\frac{\delta_0}{2}} ,
\end{align}
and if $R\ge 1$, we have 
\begin{align}
      \sum_{R'\le 1 \ \text{dyadic}} \frac{R'^{2-\frac{\delta_0}{2}} }{(R'+R)^2} + \frac{R'^{1-\frac{\delta_0}{2}} }{ R'+R } 
   \lesssim  
   \sum_{R'\le R \ \text{dyadic}} \langle R \rangle^{-1} R'^{1-\frac{\delta_0}{2}}
  \lesssim \langle R \rangle^{-1} R^{1-\frac{\delta_0}{2}}
  \lesssim  1.
\end{align}

In the region $\{r_{y_1}\ge 1\}$, we have 
\begin{align*}
 &~ \int_{\{r_{y_1}\ge 1\}}  \frac{\bar{h}\cdot v\cdot \pa v}{(r_{y_1}+R)^2} + \frac{\bar{h}\cdot \pa v\cdot \pa v}{ r_{y_1}+R }\\
\lesssim &~  \int_{\mathbb{R}_+ \times \mathbb{R}^3}   \langle r_{y_1}\rangle^{-2}|\bar{h}\cdot v\cdot \pa v| +  \langle r_{y_1}\rangle^{-1}|\bar{h}\cdot \pa v\cdot \pa v|\\
 \lesssim&~ \| \langle r\rangle^{\delta_0}\bar{h}\|_{L^\infty}\int_{\mathbb{R}_+ \times \mathbb{R}^3} \langle r_{y_1}\rangle^{-1} \langle r\rangle^{-\delta_0} | \pa v|\cdot (|\pa v| + \langle r_{y_1} \rangle^{-1} |v| )\\
 \lesssim&~ \| \langle r\rangle^{\delta_0}\bar{h}\|_{L^\infty}\int_{\mathbb{R}_+ \times \mathbb{R}^3}  \boldsymbol{1}_{\{r\ge r_{y_1}\}}\langle r_{y_1}\rangle^{-1-\delta_0}| \pa v|\cdot (|\pa v| + \langle r_{y_1} \rangle^{-1} |v| )\\
  & + \| \langle r\rangle^{\delta_0}\bar{h}\|_{L^\infty}\int_{\mathbb{R}_+ \times \mathbb{R}^3} \boldsymbol{1}_{\{r\le r_{y_1}\}}\langle r\rangle^{-1-\delta_0} | \pa v|\cdot (|\pa v| + \langle r_{y_1} \rangle^{-1} |v| ) \\
  \lesssim&~ \| \langle r\rangle^{\delta_0}\bar{h}\|_{L^\infty}\sup_{y_0}\int_{\mathbb{R}_+ \times \mathbb{R}^3}  \langle r\rangle^{-1-\delta_0} |\pa v^{y_0}| \cdot (|\pa v^{y_0}| + r^{-1}|v^{y_0}|)\\
 \lesssim&~ C \varepsilon E(v),
\end{align*}
and the last line follows by applying the bootstrap assumption:
\begin{align}
&~ \int_{\mathbb{R}_+ \times \mathbb{R}^3}  \langle r\rangle^{-1-\delta_0} |\pa v^{y_0}| \cdot (|\pa v^{y_0}| + r^{-1}|v^{y_0}|)\notag\\
=&~ \sum_{R\ge 1\ \text{dyadic}} \int_{\{R\le r\le 2R\}} \cdots +  \int_{r\le 1}  \cdots\notag\\
  \lesssim&~ \Big(\sum_{R\ge 1\ \text{dyadic}} R^{-\delta_0} + 1 \Big)C E(v)
  \lesssim  C E(v).
\end{align}
Therefore
\begin{align*}
  &~ \int_{[0,T]\times \mathbb{R}^3} \frac{ (\bar{h}\cdot v\cdot \pa v)^{y_0}}{(r+R)^2} +   \frac{(\bar{h}\cdot \pa v\cdot \pa v)^{y_0} }{r+R}\,\mathrm{d}t\mathrm{d}x \lesssim   C\varepsilon\max\{R^{-\frac{\delta_0}{2}},1\} E(v).
\end{align*}

Finally, we obtain the following bound for the right-hand side of \eqref{int}
\begin{equation}\label{le2}
  \begin{aligned}
     &~ RHS \lesssim (1+C\varepsilon)\max\{R^{-\frac{\delta_0}{2}},1\} E(v).
   \end{aligned}
\end{equation}
Combining \eqref{le1} and \eqref{le2} and letting $T\to +\infty$, we obtain
\begin{equation}
  \int_{\mathbb{R}_+\times \{R\le r\le 2R\}}e(v^{y_0})+ R^{-2} (v^{y_0})^2 \mathrm{d}t\mathrm{d}x\lesssim (1+C\varepsilon)\max\{R^{1-\frac{\delta_0}{2}},R\}  E(v).
\end{equation} 

The bootstrap argument can be closed by choosing $C$ such that 
\begin{equation}
  (1+C \varepsilon)\ll C\ \iff 1\ll C \ll \frac{1}{   \varepsilon},
\end{equation}
which is possible since $\varepsilon\ll 1$.

Applying dyadic pigeonholing and using the above estimates, we obtain the local energy estimate
\begin{align}
  &~ \int_{\mathbb{R}_+ \times \mathbb{R}^3}r^{-(1-\delta_0)}\langle r \rangle^{-2\delta_0} (e(v^{y_0})+ r^{-2}(v^{y_0})^2)\mathrm{d}t\mathrm{d}x\nonumber \\
  \lesssim&~ \sum_{R \ \text{dyadic}} \int_{\mathbb{R}_+ \times \{R\le r \le 2R \}}r^{-(1-\delta_0)}\langle r \rangle^{-2\delta_0} (e(v^{y_0})+ r^{-2}(v^{y_0})^2)\mathrm{d}t\mathrm{d}x\nonumber \\
  \lesssim&~ \Big(\sum_{R\le 1\ \text{dyadic}} R^{\frac{\delta_0}{2}}+ \sum_{R\ge  1\ \text{dyadic}} R^{-\delta_0}\Big) E(v)
  \lesssim  E(v).
\end{align}
\end{proof}

We remark that the added translations allow us to obtain similar bounds for terms involving the $L$-notation. Below we list the related estimates, which will also be needed in later sections.
\begin{corollary}Given a solution $v$ admitting the above local energy estimates and a dyadic number $\lambda\gg 1$, we have  
  \begin{align} 
    &~ \sup_{y_0} \|r^{-\frac{1-\delta_0}{2}}\langle r \rangle^{-\delta_0}v_\lambda^{y_0} \|_{L^2_{t,x}}\lesssim \lambda^{-1}E(v_\lambda)^{\frac{1}{2}},\label{c1}\\
    &~ \sup_{y_0,y_1}\|r^{-\frac{1-\delta_0}{2}}\langle r \rangle^{-\delta_0}r^{-1}_{y_1} v^{y_0} \|_{L^2_{t,x}}\lesssim E(v)^{\frac{1}{2}},\label{c2}\\
    &~ \sup_{y_0,y_1}\|\langle r\rangle^{-s_2} \langle r_{y_1}\rangle^{-s_1} \pa v^{y_0} \|_{L^2_{t,x}}\lesssim E(v)^{\frac{1}{2}},\ \text{for}\ s_1+s_2\ge  \frac{1+\delta_0}{2}.\label{c3}
   \end{align}
\end{corollary}
\begin{proof}
We begin with \eqref{c1}, namely
\begin{equation}
  \sup_{y_0} \|r^{-\frac{1-\delta_0}{2}}\langle r \rangle^{-\delta_0}v_\lambda^{y_0} \|_{L^2_{t,x}}\lesssim \lambda^{-1}E(v_\lambda)^{\frac{1}{2}}.
\end{equation}

Using the $L$-notation, the left-hand side can be rewritten as
\begin{equation}
 \lambda^{-2} \int_{\mathbb{R}_+\times\mathbb{R}^3} L(r^{-(1-\delta_0)}\langle r \rangle^{-2\delta_0},   \pa_x v_\lambda,  \pa_x v_\lambda)\mathrm{d}t \mathrm{d}x
\end{equation}
since $\lambda\gg 1$. By the properties of the $L$-notation and the local energy estimates, we obtain 
\begin{equation}
  \lesssim \lambda^{-2}\sup_{y_\alpha}\int_{\mathbb{R}_+\times\mathbb{R}^3} r^{-(1-\delta_0)}\langle r \rangle^{-2\delta_0}  |\pa_x v_\lambda^{y_0}| |\pa_x v_\lambda^{y_1}| \mathrm{d}t \mathrm{d}x\lesssim\lambda^{-2}  E(v_\lambda).
\end{equation}

For \eqref{c2}, we split the region into $\{r_{y_1}\ge r\}$ and $\{r_{y_1}\le r\}$ as before, and we have
\begin{align*}
  &~ \sup_{y_\alpha}\|r^{-\frac{1-\delta_0}{2}}\langle r \rangle^{-\delta_0}r^{-1}_{y_1} v^{y_0} \|_{L^2_{t,x}}\\
   \le&~ \sup_{y_\alpha}\|\boldsymbol{1}_{\{r_{y_1}\ge r\}}r^{-\frac{3-\delta_0}{2}}\langle r \rangle^{-\delta_0}  v^{y_0} \|_{L^2_{t,x}}   +\sup_{y_\alpha}\|\boldsymbol{1}_{\{r_{y_1}\le r\}}r_{y_1}^{-\frac{3-\delta_0}{2}}\langle r_{y_1} \rangle^{-\delta_0}  v^{y_0} \|_{L^2_{t,x}}\\
  \le&~ \sup_{y_\alpha}\|r^{-\frac{3-\delta_0}{2}}\langle r \rangle^{-\delta_0}v^{y_0} \|_{L^2_{t,x}} + \|r^{-\frac{3-\delta_0}{2}}\langle r \rangle^{-\delta_0}v^{y_0-y_1} \|_{L^2_{t,x}}\\
  \lesssim&~ E(v)^{\frac{1}{2}}.
\end{align*}
Estimate \eqref{c3} follows from similar arguments.
\end{proof}

\section{A new type of div-curl lemma and null form estimates}
% We also need the frequency localized version of local energy estimates.
% \begin{proposition}
%   If the solution $v$ to \eqref{model} is localized in Fourier region $|\xi|\le  \lambda$, then we have 
% \end{proposition}
% \begin{proof}
%   By assumption $e(v)$ is also localized in the Fourier region $|\xi|\le 4\lambda$. Applying the Plancherel theorem, we have 
%   \begin{align*}
%     \int_{\mathbb{R}^3} r^{-2s} e(v) \mathrm{d}x& = C\int_{\mathbb{R}^3}|\xi|^{2s-3} \chi((4\lambda)^{-1}\xi)\widehat{e(v)}(\xi)\mathrm{d}\xi\\
%     & = C\int_{\mathbb{R}^3}|\xi|^{2s-3} \chi((4\lambda)^{-1}\xi)\widehat{e(v)}(\xi)\mathrm{d}\xi
%   \end{align*}
% \end{proof}
In this section we derive a new type of div-curl lemma and use it to obtain the core space-time $L^2_{t,x}$ estimates of this paper.
\subsection{A new type of div-curl lemma}
The detailed proof of the following lemma can be found in \cite{2024Physical,wang_2023}; we repeat it
 here for the convenience of the reader.
\begin{lemma}\label{dc}
	Suppose that
	\begin{equation}
	\left\{
	\begin{aligned}
	&~ f_t^{11} + f_r^{12} =G^1\\
	&~ f_t^{21}-f_r^{22}=G^2
	 \end{aligned}
	\right.
	\end{equation}
	
	\begin{equation}
	\begin{aligned}
	&~ f^ {12} \rightarrow 0,\ r\rightarrow 0.
	 \end{aligned}
	\end{equation}
	Then there holds
	\begin{equation}
	\begin{aligned}
	 &~ \int_{[0,T]\times \mathbb{R}_+} f^{11}f^{22}+f^{12}f^{21}\\
	 \lesssim&~ \Big(  \sup\limits_{0\leq t\leq T}  \|f^{11}\left(t\right)\|_{L^1} + \int_{[0,T]\times \mathbb{R}_+}\lvert G^1\rvert \Big)\\
	 \quad\cdot&~ \Big( \sup\limits_{0\leq t\leq T} \|f^{21}\left(t\right)\|_{L^1}  +\int_{[0,T]\times \mathbb{R}_+}\lvert G^2\rvert \Big)
	 \end{aligned}
	\end{equation}
	provided that the right hand side is bounded.
\end{lemma}

\begin{proof}
 
	We note that 
	\begin{equation}\label{dd1}
	f^{21}\int_{0}^{r} f_t^{11}+f^{12}f^{21}=\int_{0}^{r}G^1f^{21},
	\end{equation}
	
	\begin{equation}\label{dd2}
	f^{21}_t\int_{0}^{r} f^{11}-f_r^{22}\int_{0}^{r}f^{11}= G^2\int_{0}^{r}f^{11}.
	\end{equation}
    
    \eqref{dd1}+\eqref{dd2}:
	\begin{equation}
	\int_{0}^{\infty} \left( f^{21}\int_{0}^{r}f^{11}\right)_t+\int_{0}^{\infty} 
	\left(f^{12}f^{21}-f_r^{22}\int_{0}^{r}f^{11}\right)=\int_{0}^{\infty}\left(f^{21}\int_{0}^{r}G^1 + G^2\int_{0}^{r}f^{11}\right).
	\end{equation}
	
	Integrating the first term in $t$ over $[0,T]$ and integrating by parts the term involving $f_r^{22}$, we obtain
	\begin{equation}
	\begin{aligned}
	&~ \int_{[0,T]\times \mathbb{R}_+} f^{11}f^{22}+f^{12}f^{21}= \int_{0}^{\infty} \left( f^{21}\int_{0}^{r}f^{11}\right)\left(0\right)- \left(f^{21}\int_{0}^{r} f^{11}\right)\left(T\right) \\
	+&~ \int_{[0,T]\times \mathbb{R}_+}\left(f^{21}\int_{0}^{r}G^1 + G^2\int_{0}^{r}f^{11}\right)\\
	:=&~ \mathcal{A}_1+\mathcal{A}_2+\mathcal{A}_3,
	 \end{aligned}
	\end{equation}
where	
\begin{align*}
  &~ \lvert \mathcal{A}_1\rvert \lesssim \|f^{11}\left(0\right)\|_{L^1}\|f^{21}\left(0\right)\|_{L^1} + \|f^{11}\left(T\right)\|_{L^1}\|f^{21}\left(T\right)\|_{L^1},\\
  &~ \lvert \mathcal{A}_2\rvert
  \lesssim  \int_{0}^{T} \|f^{21}\left(t\right)\|_{L^1}\|G^1\left(t\right)\|_{L^1}
  \lesssim  \sup\limits_{0\leq t\leq T} \| f^{21}\left(t\right)\|_{L^1}\left(\int_{[0,T]\times \mathbb{R}_+}\lvert G^1\rvert\right),\\
  &~ \lvert \mathcal{A}_3\rvert
  \lesssim  \int_{0}^{T} \|f^{11}\left(t\right)\|_{L^1}\|G^2\left(t\right)\|_{L^1}
  \lesssim  \sup\limits_{0\leq t\leq T} \| f^{11}\left(t\right)\|_{L^1}\left(\int_{[0,T]\times \mathbb{R}_+}\lvert G^2\rvert\right).
\end{align*}
	Based on the above analysis, we complete the proof. 
\end{proof}

\subsection{Null form estimates}
We will use the above lemma to establish space-time bilinear estimates for frequency-localized solutions $v_\lambda,w_{\mu,\nu}$ to the equations
\begin{align}
 &~ \Box_{\bar{g}^{y_0}} v^{y_0}_\lambda = [N(v_\lambda)]^{y_0},\label{eq-v}\\
 &~ \Box_{\bar{g}} w_{\mu,\nu}   = N(w_{\mu,\nu}).\label{eq-w}
\end{align}
More precisely, we will prove
\begin{proposition}\label{dcnull}
  Under the same assumption in Proposition \ref{Le},  given frequency localized Schwartz solutions $v_\lambda,w_{\mu,\nu}$ satisfying the local energy estimates and sufficiently small number $\delta_0>0$,
  we have
  \begin{equation}\label{4.2RHS}
    \begin{aligned}
    &~ \int_{\mathbb{R}_+\times \{r\ge 1\}} r^{3+\delta_0}  \Big(\ins e_\omega(v^{y_0}_{\lambda})\ins e(w_{\mu,\nu})+\ins e(v^{y_0}_{\lambda})\ins e_\omega(w_{\mu,\nu})+ \ins e_{\pm}(v^{y_0}_{\lambda})\ins e_{\mp}(w_{\mu,\nu})\Big)\mathrm{d}t\mathrm{d}r\\
      \lesssim&~ (\min\{\lambda,\mu\} +  \min\{\lambda,\mu\}^{1-2\delta_0}\nu^{2\delta_0}) E(v_\lambda)E(w_{\mu,\nu}).
     \end{aligned}
  \end{equation} 
\end{proposition}
\begin{proof}
  We recall the  energy balance law for $v^{y_0}_{\lambda}$
\begin{equation}\label{energy-v}
  \begin{aligned}
  &~ \pa_t\ins \big( e(v_\lambda^{y_0})+(\bar{h}\cdot \pa v_\lambda\cdot \pa v_\lambda)^{y_0}\big) r^2\mathrm{d}\omega\\
  &- \pa_r\ins \big( p(v_\lambda^{y_0})+(\bar{h}\cdot \pa v_\lambda\cdot \pa v_\lambda)^{y_0}\big) r^2\mathrm{d}\omega\\
 =&~ \ins  \mathbf{F}_e(v_\lambda^{y_0})\  r^2\mathrm{d}\omega,
   \end{aligned}
\end{equation}
Let $\beta(r)$ be a smooth cutoff to the region $[1,+\infty)$, whose derivative $\beta'(r)$ is supported in $[\frac{1}{2},\frac{3}{2}]$. Multiplying \eqref{momentum} by $\beta(r)r^{-1+\delta_0}$ and integrating by parts, we obtain the weighted momentum balance law for $w_{\mu,\nu}$
\begin{equation}\label{momentum-w}
  \begin{aligned}
    &~ \pa_t\ins \beta(r)r^{ 1+\delta_0} [ p(w_{\mu,\nu})+\bar{h}\cdot \pa w_{\mu,\nu}\cdot \pa w_{\mu,\nu}+r^{-1}(1+\bar{h})\cdot w_{\mu,\nu}\cdot \pa w_{\mu,\nu}]\\
    &-  \pa_r\ins  \beta(r)r^{ 1+\delta_0}[e(w_{\mu,\nu})-2e_\omega(w_{\mu,\nu})+\bar{h}\cdot \pa w_{\mu,\nu}\cdot \pa w_{\mu,\nu}+  r^{-1}(1+\bar{h})\cdot w_{\mu,\nu}\cdot \pa w_{\mu,\nu}+ r^{-2} w_{\mu,\nu}^2] \\
    =&~ \ins \Big[(\beta(r)r^{\delta_0}+\beta'(r)r^{1+\delta_0})\big((1+\bar{h})\cdot \pa w_{\mu,\nu}\cdot \pa w_{\mu,\nu} +r^{-1}(1+\bar{h})\cdot w_{\mu,\nu}\cdot \pa w_{\mu,\nu}+ r^{-2}w_{\mu,\nu}^2\big)\Big]\\
    &  + \ins \beta(r)r^{ 1+\delta_0}\mathbf{F}_m(w_{\mu,\nu}) .
   \end{aligned}
\end{equation}

Since $\ins r^2(p(v_\lambda)+\bar{h}\cdot \pa v_\lambda\cdot \pa v_\lambda)\to 0$ as $r\to 0$, applying Lemma \ref{dc} to \eqref{energy-v} and \eqref{momentum-w}, we have 
\begin{equation}\label{4.2con}
  \begin{aligned}
   &~ \int_{[0,T]\times \mathbb{R}_+}\beta(r) r^{3+\delta_0} \Big(\ins e(v^{y_0}_{\lambda})\ins e(w_{\mu,\nu})-\ins p(v^{y_0}_{\lambda})\ins p(w_{\mu,\nu})\Big)\mathrm{d}t\mathrm{d}r \\
   \lesssim&~ E(v_\lambda)( E(w_{\mu,\nu})+ \mathbf{Err}_1)+ \mathbf{Err}_2+ \mathbf{Err}_3,
   \end{aligned}
\end{equation}
where
\begin{align}
  &~ \mathbf{Err}_1 = \int_{[0,T]\times \mathbb{R}^3}|(1+\bar{h})\cdot \pa w_{\mu,\nu}\cdot \pa w_{\mu,\nu} +r^{-1}(1+\bar{h})\cdot w_{\mu,\nu}\cdot \pa w_{\mu,\nu} +r^{-2}w_{\mu,\nu}^2|\nonumber\\
  &~ \qquad  \qquad \qquad\quad \cdot |\beta(r)r^{-2+ \delta_0}+\beta'(r)r^{-1+\delta_0}| \mathrm{d}t \mathrm{d}x,\\
  &~ \mathbf{Err}_2 =\int_{[0,T]\times \mathbb{R}_+}\Big(\ins(\pa v^{y_0}_{\lambda}\cdot \pa v^{y_0}_{\lambda} )   \ins(\bar{h}\cdot \pa w_{\mu,\nu}\cdot \pa w_{\mu,\nu} )\notag\\
  &\qquad  \qquad \qquad\qquad +\ins (\bar{h}\cdot \pa v_\lambda\cdot \pa v_\lambda)^{y_0}\ins  (\pa w_{\mu,\nu}\cdot \pa w_{\mu,\nu})\notag\\
  &\qquad  \qquad \qquad\qquad +\ins  ((1+\bar{h})\cdot \pa v_\lambda\cdot \pa v_\lambda)^{y_0}\ins (1+\bar{h})\cdot r^{-1}w_{\mu,\nu}\cdot  \pa w_{\mu,\nu}   \notag\\
  &\qquad  \qquad \qquad\qquad +\ins ( (1+\bar{h})\cdot \pa v_\lambda\cdot \pa v_\lambda)^{y_0}\ins  r^{-2}w_{\mu,\nu}^2 \Big) \beta(r) r^{3+\delta_0} \mathrm{d}t \mathrm{d}r\\
  &~ \mathbf{Err}_{3} = \int_{[0,T]\times \mathbb{R}_+} \Big(\ins e(v^{y_0}_{\lambda})\ins e_\omega(w_{\mu,\nu})\Big)  \beta(r)r^{3+\delta_0}\mathrm{d}t \mathrm{d}r.
\end{align}

We note that the left-hand side of \eqref{4.2con} differs from the expression produced by Lemma \ref{dc} by the term $2\mathbf{Err}_3$. Indeed, since the flux $f^{22}$ in \eqref{momentum-w} contains $e(w_{\mu,\nu})-2e_\omega(w_{\mu,\nu})$, Lemma \ref{dc} yields $e(v^{y_0}_\lambda)[e(w_{\mu,\nu})-2e_\omega(w_{\mu,\nu})]-p(v^{y_0}_\lambda)p(w_{\mu,\nu})$, and the difference from the left-hand side of \eqref{4.2con} is exactly $2e(v^{y_0}_\lambda)e_\omega(w_{\mu,\nu})$. This is why the term $\mathbf{Err}_3$ appears on the right-hand side of \eqref{4.2con}; it will be controlled in the last step of the proof.

Direct calculation shows that the left-hand side of \eqref{4.2con} satisfies
\begin{align*}
   LHS =&   \int_{[0,T]\times \mathbb{R}_+} \Big(\ins (e_+(v^{y_0}_{\lambda})+e_-(v^{y_0}_{\lambda})+e_\omega(v^{y_0}_{\lambda}))\ins (e_+(w_{\mu,\nu})+e_-(w_{\mu,\nu})+e_\omega(w_{\mu,\nu}))\\
  &  \qquad\qquad -\ins (e_+(v^{y_0}_{\lambda})-e_-(v^{y_0}_{\lambda}))\ins (e_+(w_{\mu,\nu})-e_-(w_{\mu,\nu}))\Big)\beta(r) r^{3+\delta_0}\mathrm{d}t\mathrm{d}r\\
  =&  \int_{[0,T]\times \mathbb{R}_+} \Big(\ins e_\omega(v^{y_0}_{\lambda})\ins e(w_{\mu,\nu})+\ins e(v^{y_0}_{\lambda})\ins e_\omega(w_{\mu,\nu})+ \\
   &~ \qquad\qquad +2\ins e_{+}(v^{y_0}_{\lambda})\ins e_{-}(w_{\mu,\nu})+2\ins e_{-}(v^{y_0}_{\lambda})\ins e_{+}(w_{\mu,\nu})\Big)\beta(r) r^{3+\delta_0}\mathrm{d}t\mathrm{d}r.
\end{align*}

It remains to control the error terms on the right-hand side. The first error term can be controlled by the local energy estimates, namely
\begin{align}
  &~ \mathbf{Err}_1 \lesssim \int_{[0,T]\times \{r\ge \frac{1}{2}\}} \langle r\rangle^{-2+\delta_0} (e(w_{\mu,\nu})+r^{-2}w^2_{\mu,\nu})\mathrm{d}t \mathrm{d}x \lesssim E(w_{\mu,\nu}).
\end{align}

For the second error term, we have  
\begin{align*}
  &~ \mathbf{Err}_{2} \lesssim\int_{[0,T]\times\{r\ge  \frac{1}{2}\}}\Big(\ins (\bar{h}\cdot \pa v_\lambda\cdot \pa v_\lambda)^{y_0}\ins  |\pa w_{\mu,\nu}|^2  \\
  &\qquad\qquad\qquad\qquad\ + \ins |\pa v^{y_0}_{\lambda}|^2  \ins(\bar{h}\cdot \pa w_{\mu,\nu}\cdot \pa w_{\mu,\nu} )\\
  &\qquad\qquad\qquad\qquad\ + \ins |\pa v^{y_0}_{\lambda}|^2  \ins r^{-1}|w_{\mu,\nu}\cdot \pa w_{\mu,\nu}|   +r^{-2}w_{\mu,\nu}^2 \Big) r^{3+\delta_0} \mathrm{d}t \mathrm{d}r\\
  &~ =: \mathbf{Err}_{21}+ \mathbf{Err}_{22}+ \mathbf{Err}_{23}.
\end{align*}

For $\mathbf{Err}_{21}$ and $\mathbf{Err}_{22}$, it is sufficient to estimate the first term, since the second one can be handled similarly. We expand it as 
\begin{align}
  &~ \mathbf{Err}_{21} \lesssim \|\langle r_{y_0}\rangle^{2\delta_0}\bar{h}^{y_0}\|_{L^\infty}\|r \pa v_\lambda^{y_0}\|_{L^\infty_r L^2_\omega}^2\int_{\mathbb{R}_+ \times \mathbb{R}^3} \langle r_{y_0}\rangle^{-2\delta_0} \langle r\rangle^{-1+\delta_0}|\pa w_{\mu,\nu}|^2 \mathrm{d}t \mathrm{d}x,
\end{align}
or alternatively
\begin{align}
  &~ \mathbf{Err}_{21} \lesssim \|\langle r_{y_0}\rangle^{2\delta_0}\bar{h}^{y_0}\|_{L^\infty}\|r \pa w_{\mu,\nu}\|_{L^\infty_r L^2_\omega}^2\int_{\mathbb{R}_+ \times \mathbb{R}^3} \langle r_{y_0}\rangle^{-2\delta_0} \langle r\rangle^{-1+\delta_0}|\pa v_\lambda^{y_0}|^2 \mathrm{d}t \mathrm{d}x.
\end{align}
Using the smallness assumptions and the trace-type Hardy inequalities \eqref{2.6}, we have
\begin{align*}
  & \| r_{y_0}^{2\delta_0}\bar{h}^{y_0}\|_{L^\infty}=\| r^{2\delta_0}\bar{h}  \|_{L^\infty}
  \lesssim  \|  \bar{h}  \|^{1-2\delta_0}_{L^\infty}\| r\Lambda_\omega^{1+\delta_0}\bar{h}  \|^{2\delta_0}_{L^\infty_r L^2_\omega}
  \lesssim  \|  \bar{h}  \|^{1-2\delta_0}_{L^\infty}\|\Lambda_\omega^{1+\delta_0}\bar{h}  \|^{2\delta_0}_{H^{\frac{1}{2}+2\delta_0}}
  \le  \varepsilon,\\
  & \|r \pa w_{\mu,\nu}\|_{L^\infty_r L^2_\omega}
  \lesssim  \mu^{\frac{1}{2}}\|\pa w_{\mu,\nu}\|_{L^2}
  \le  \mu^{\frac{1}{2}}E(w_{\mu,\nu})^{\frac{1}{2}},\\
  &  \|r \pa v_{\lambda}^{y_0}\|_{L^\infty_r L^2_\omega}
  \lesssim \lambda^{\frac{1}{2}}\|\pa v_{\lambda}^{y_0}\|_{L^2}
  \le  \lambda^{\frac{1}{2}}E( v_{\lambda})^{\frac{1}{2}}.
\end{align*}
The first line combines the $L^\infty$-smallness of $\bar{h}$ with the angular trace-type inequality \eqref{2.6}, and the remaining two lines follow directly from \eqref{2.6}.
Together with \eqref{c3}, we obtain
\begin{equation}
  \mathbf{Err}_{21}\lesssim \min\{\lambda,\mu\} E(w_{\mu,\nu})E(v_\lambda).
\end{equation}
 
For $\mathbf{Err}_{23}$, when $\mu\lesssim \lambda$ we have 
\begin{align*}
  &~ \mathbf{Err}_{23} \lesssim  (\|r^{2\delta_0} w_{\mu,\nu}\|_{L_r^\infty L^2_{\omega}} \|r \pa w_{\mu,\nu}\|_{L_r^\infty L^2_{\omega}}+ \|r^{2\delta_0} w_{\mu,\nu}\|_{L_r^\infty L^2_{\omega}}^2) \int_{\mathbb{R}_+ \times \mathbb{R}^3} \langle r\rangle^{-1-\delta_0}(\pa v_\lambda^{y_0})^2 \mathrm{d}t \mathrm{d}x\\
  \lesssim&~ \mu^{\frac{1}{2}}\|r^{2\delta_0} w_{\mu,\nu}\|_{L_r^\infty L^2_{\omega}}E(w_{\mu,\nu})^{\frac{1}{2}} E(v_\lambda).
\end{align*}
Notice that 
\begin{align*}
 &~ \| r^{2\delta_0} w_{\mu,\nu}  \|^2_{L^\infty_r L^2_\omega} = \ins r^{4\delta_0} w_{\mu}^2\\
 \lesssim&~ \int_{r}^{+\infty}\ins\Big(\frac{r}{\rho}\Big)^{4\delta_0}\rho^{-2+4\delta_0}|w_{\mu,\nu}\pa_r w_{\mu,\nu} |\rho^2\mathrm{d}\rho \mathrm{d}\omega \\
  \lesssim&~ \| r^{-1+4\delta_0}\pa_r w_{\mu,\nu}   \|_{L^2}  \| r^{-1} w_{\mu,\nu}   \|_{L^2} \\
  \lesssim&~ \| \pa_r w_{\mu,\nu} \|_{\dot{H}^{1-4\delta_0}}\| w_{\mu,\nu} \|_{\dot{H}^{1}}
  \lesssim\mu^{1-4\delta_0}E(w_{\mu,\nu}),
\end{align*}
Here we used the Hardy inequality, the Bernstein inequality, and $\|w_{\mu,\nu}\|_{\dot H^1}\lesssim E(w_{\mu,\nu})^{\frac{1}{2}}$.
Therefore 
\begin{equation}
  \mathbf{Err}_{23}\lesssim \mu^{1-2\delta_0}E(w_{\mu,\nu}) E(v_\lambda).
\end{equation}

When $\mu\gg \lambda \gtrsim 1$, we have 
\begin{align*}
  &~ \mathbf{Err}_{23}\\
  \lesssim&~ \|r\pa v^{y_0}_{\lambda}\|_{L^\infty_r L_\omega^2}^2 \int_{\mathbb{R}_+ \times \mathbb{R}^3} \langle r \rangle^{-2+\delta_0}\lambda^{-1}L(\pa_x w_{\mu,\nu},\pa w_{\mu,\nu})   +\langle r \rangle^{-3+\delta_0}\lambda^{-2}L(\pa_x w_{\mu,\nu},\pa_x w_{\mu,\nu})   \mathrm{d}t \mathrm{d}x\\
  \lesssim&~ \lambda E(v_\lambda) E(w_{\mu,\nu}),
\end{align*} 
where we use \eqref{c1}.

Combining the above estimates, we have 
\begin{equation}
  \mathbf{Err}_2 \lesssim \min\{\lambda,\mu\} E(v_\lambda)E(w_{\mu,\nu}).
\end{equation}

For the third error term, we start with the case of $\mu\lesssim \lambda$,
\begin{equation*}
  \mathbf{Err}_3  \lesssim \| r^{1+\delta_0}\slashed{\pa} w_{\mu,\nu}\|^2_{L_r^\infty L_\omega^2} \int_{[0,T]\times \{r\ge \frac{1}{2}\}}\langle r\rangle^{-1-\delta_0} e(v^{y_0}_{\lambda}) \mathrm{d}t \mathrm{d}x\lesssim E(v_\lambda ) \| r^{\delta_0}\pa_\omega w_{\mu,\nu}\|^2_{L_r^\infty L_\omega^2}.
\end{equation*}

It is clear that 
\begin{align*}
  &~ \| r^{\delta_0}\pa_\omega w_{\mu,\nu}\|^2_{L_r^\infty L_\omega^2} \le \int_{r}^{+\infty}\ins\Big(\frac{r}{\rho}\Big)^{2\delta_0} \rho^{-2+2\delta_0}|\pa_\omega w_{\mu,\nu}| |\pa_\omega \pa_r w_{\mu,\nu}|\rho^2\mathrm{d}\rho\mathrm{d}\omega \\
  \le&~ \int_{\mathbb{R}^3} |\slashed{\pa} w_{\mu,\nu}||\pa_\omega \pa_r w_{\mu,\nu}|^{2\delta_0} |\slashed{\pa}\pa_ r w_{\mu,\nu}|^{1-2\delta_0} \mathrm{d}x \\
  \lesssim&~ \mu^{1-2\delta_0}\nu^{2\delta_0} E(w_{\mu,\nu}),
\end{align*}
where we use the Bernstein inequality
$$\|\pa_\omega \pa_r w_{\mu,\nu}\|_{L^2_\omega}\lesssim \nu \|\pa_r w_{\mu,\nu}\|_{L^2_\omega},$$
and thus 
\begin{equation}
  \mathbf{Err}_{3} \lesssim \mu^{1-2\delta_0}\nu^{2\delta_0}E(v_\lambda)E(w_{\mu,\nu}).
\end{equation}

In the case of $\mu\gg \lambda \gtrsim 1$, we have 
\begin{align}
  &~ \mathbf{Err}_3 \lesssim \| r\pa v_{\lambda}\|^2_{L_r^\infty L_\omega^2} \int_{[0,T]\times \{r\ge \frac{1}{2}\}}r^{-1+\delta_0} e_\omega(w_{\mu,\nu}) \mathrm{d}t \mathrm{d}x\notag\\
  \lesssim&~ \lambda E(v_\lambda ) \int_{[0,T]\times \{r\ge \frac{1}{2}\}}r^{-1+\delta_0} e_\omega(w_{\mu,\nu}) \mathrm{d}t \mathrm{d}x.
\end{align}

Applying \eqref{c1}, we have
\begin{align*}
  &~ \int_{[0,T]\times \{r\ge \frac{1}{2}\}} r^{-1+\delta_0} e_\omega(w_{\mu,\nu})  \mathrm{d}t \mathrm{d}x= \int_{[0,T]\times \{r\ge \frac{1}{2}\}}r^{-1+\delta_0} \|\slashed{\pa}w_{\mu,\nu}\|^2_{L^2_\omega}  r^2\mathrm{d}t \mathrm{d}r\\
  \le&~ \int_{[0,T]\times \{r\ge \frac{1}{2}\}}r^{-1-\delta_0} \|\pa_\omega w_{\mu,\nu}\|^{2\delta_0}_{L^2_\omega}\| \pa  w_{\mu,\nu}\|^{2-2\delta_0}_{L^2_\omega}  r^2\mathrm{d}t \mathrm{d}r\\
  \lesssim&~ \nu^{2\delta_0}\Big(\int_{[0,T]\times \{r\ge \frac{1}{2}\}}r^{-1-\delta_0}  |w_{\mu,\nu} |^{2 }\mathrm{d}t\mathrm{d}x \Big)^{\delta_0}\Big(\int_{[0,T]\times \{r\ge \frac{1}{2}\}}r^{-1-\delta_0}  |\pa w_{\mu,\nu} |^{2 }\mathrm{d}t \mathrm{d}x\Big)^{ 1-\delta_0}\\
  \lesssim&~ \mu^{-2\delta_0}\nu^{2\delta_0}E(w_{\mu,\nu})^{1-\delta_0}\sup_{y_0}\Big(\int_{[0,T]\times \{r\ge \frac{1}{2}\}}r^{-1-\delta_0}  |(\pa_x w_{\mu,\nu})^{y_0} |^{2 }\mathrm{d}t\mathrm{d}x \Big)^{\delta_0}\\
  \lesssim&~ \mu^{-2\delta_0}\nu^{2\delta_0}E(w_{\mu,\nu}).
\end{align*}
Here we used the Bernstein inequality $\|\pa_\omega w_{\mu,\nu}\|_{L^2_\omega}\lesssim\nu\|w_{\mu,\nu}\|_{L^2_\omega}$, H\"older's inequality, and \eqref{c1} for the factor containing $|w_{\mu,\nu}|$.
Thus we have
\begin{equation}
  \mathbf{Err}_{3} \lesssim \lambda\mu^{-2\delta_0}\nu^{2\delta_0}E(w_{\mu,\nu}).
\end{equation}

Combining the above estimates for different cases, we have 
\begin{equation}
  \mathbf{Err}_{3} \lesssim \min\{\lambda,\mu\}^{1-2\delta_0}\nu^{2\delta_0}E(v_\lambda)E(w_{\mu,\nu}).
\end{equation}

It is clear that the error terms are bounded by the right-hand side in \eqref{4.2RHS}. Letting $T\to +\infty$ in \eqref{4.2con}, we obtain the desired estimates.
\end{proof}

As a consequence of Propositions \ref{Le} and \ref{dcnull}, we have the following null form estimates.
\begin{proposition}\label{null}
  Under the same assumption in Proposition \ref{Le}, we have the following $L^2_{t,x}$ estimates
 \begin{align}
  &~ \sup_{y_0}\| r^{\frac{1-\delta_0}{2}}\langle r \rangle^{\delta_0} G(w_{\lambda_1,\nu},v^{y_0}_{\lambda_2})  \|_{L^2_{t,x}}\lesssim \nu^{1+\delta_0}\min\{\lambda_1,\lambda_2\}^{\frac{1}{2}+\delta_0}E(  w_{\lambda_1,\nu})^{\frac{1}{2}}E(  v_{\lambda_2})^{\frac{1}{2}}.\label{4.31}
  \end{align} 
In particular, we have the following $L^2_{t,x}$ estimates 
\begin{equation}\label{4.35}
  \begin{aligned}
  &~ \| r^{\frac{1-\delta_0}{2}}\langle r \rangle^{\delta_0} h^{i\alpha\beta }[R_\nu,\pa_\beta u_{\lambda_1,\nu_1} ]\pa_i\pa_{\alpha } u_{\lambda_2,\nu_2} \|_{L^2_{t,x}} \\
   \lesssim&~  \min\{\nu^{-1}\nu_1,1\}\min\{\nu_1,\nu_2\}^{1+\delta_0} \min\{\lambda_1,\lambda_2\}^{\frac{1}{2}+\delta_0} \lambda_2 E( u_{\lambda_1,\nu_1})^{\frac{1}{2}}E(u_{\lambda_2,\nu_2})^{\frac{1}{2}},
   \end{aligned}
\end{equation}
and we may replace the bound on the right-hand side by the interpolated estimate
\begin{equation}\label{4.29}
  \min\{\nu^{-1}\nu_1,1\}\nu_{1}^{1+\delta_0} \min\{\lambda_1,\lambda_2\}^{\frac{1}{2}+\delta_0} \lambda_2 E( u_{\lambda_1,\nu_1})^{\frac{1}{2}}E(u_{\lambda_2,\nu_2})^{\frac{s}{2}}E(u_{\lambda_2})^{\frac{1-s}{2}},\ \text{for}\ 0\le s\le 1.
\end{equation}

 Moreover, given solution $v_1$ admitting the local energy estimates, we have the  following $L^1_{t,x}$ estimates 
 \begin{align}
&  \sup_{y_\alpha}\|(\pa v_1^{y_0}+ r_{y_1}^{-1} v_1^{y_0})   \cdot G(w_{\lambda_1,\nu},v^{y_2}_{\lambda_2})  \|_{L^1_{t,x}} \label{4.33}\\
 &\qquad \lesssim \nu^{1+\delta_0}\min\{\lambda_1,\lambda_2\}^{\frac{1}{2}+\delta_0}E(v_1 )^{\frac{1}{2}}E( w_{\lambda_1,\nu})^{\frac{1}{2}}E(  v_{\lambda_2})^{\frac{1}{2}},\notag\\
& \sup_{y_\alpha}\|(\pa v_1^{y_0}+ r_{y_1}^{-1}v_1^{y_0})\cdot G(w_{\lambda_1,\nu_1},(v_{\lambda_2,\nu_2})^{y_2})  \|_{L^1_{t,x}} \label{4.34}\\
&\qquad\lesssim\min\{\nu_1,\nu_2\}^{1+\delta_0}\min\{\lambda_1,\lambda_2\}^{\frac{1}{2}+\delta_0}E(v_1 )^{\frac{1}{2}}E( w_{\lambda_1,\nu_1})^{\frac{1}{2}}E(  v_{\lambda_2,\nu_2})^{\frac{1}{2}}.\notag
\end{align} 
\end{proposition} 
\begin{proof}
 We start by proving \eqref{4.31}. 
 In the region $\{r\le 1\}$, we have 
  \begin{align}
    &~ \| \boldsymbol{1}_{\{r\le 1\}} r^{\frac{1-\delta_0}{2}} \pa w_{\lambda_1,\nu} \cdot \pa v_{\lambda_2}^{y_0}  \|_{L^2_{t,x}} \lesssim \big\|\| \pa w_{\lambda_1,\nu}\|_{  L^\infty_\omega} \|   \pa v_{\lambda_2}^{y_0}  \|_{L^2_{\omega}} \boldsymbol{1}_{\{r\le 1\}}r^{\frac{3-\delta_0}{2}} \big\|_{L^2_{t,r}}\nonumber\\
    \lesssim&~ \nu  \big\|\|  \pa w_{\lambda_1,\nu}\|_{L^2_\omega} \|   \pa v_{\lambda_2}^{y_0}  \|_{L^2_{\omega}} \boldsymbol{1}_{\{r\le 1\}}r^{\frac{3-\delta_0}{2}} \big\|_{L^2_{t,r}}\nonumber \\
    \lesssim&~ \nu \|r^{1-\delta_0}\pa w_{\lambda_1,\nu}\|_{L^\infty_r L^2_\omega} \|\boldsymbol{1}_{\{r\le 1\}} r^{-\frac{1-\delta_0}{2}}  \pa v_{\lambda_2}^{y_0}\|_{L^2_{t,x}}\nonumber\\
    \lesssim&~ \nu\lambda_1^{\frac{1}{2}+\delta_0} E( w_{\lambda_1,\nu})^{\frac{1}{2}}E(v_{\lambda_2})^{\frac{1}{2}}.
   \end{align}
Switching the roles of $w_{\lambda_1,\nu}$ and $v_{\lambda_2}^{y_0}$ in the last step, we also have
\begin{equation}
  \| \boldsymbol{1}_{\{r\le 1\}} r^{\frac{1-\delta_0}{2}} \pa w_{\lambda_1,\nu} \cdot \pa v_{\lambda_2}^{y_0}  \|_{L^2_{t,x}} \lesssim \nu\lambda_2^{\frac{1}{2}+\delta_0} E( w_{\lambda_1,\nu})^{\frac{1}{2}}E(v_{\lambda_2})^{\frac{1}{2}}.
\end{equation}

  In the region $\{r\ge 1\}$, we 
apply Proposition \ref{dcnull} and obtain
\begin{align}
  &~ \| \boldsymbol{1}_{\{r\ge 1\}} r^{\frac{1+\delta_0}{2}} G( w_{\lambda_1,\nu},v_{\lambda_2}^{y_0})  \|^2_{L^2_{t,x}}\notag\\
  \lesssim&~ \nu^2\int_{\mathbb{R}_+ \times \{r\ge 1\}}\big( \|D_\pm w_{\lambda_1,\nu}\|^2_{L_\omega^2} \|D_{\mp }v_{\lambda_2}^{y_0}\|^2_{L_\omega^2} + \| \pa  w_{\lambda_1,\nu}\|^2_{L_\omega^2} \| \slashed{\pa}v_{\lambda_2}^{y_0}\|^2_{L_\omega^2}\notag\\
  &~ \quad \quad\quad\quad\quad\quad  +\|\slashed{\pa} w_{\lambda_1,\nu}\|^2_{L_\omega^2} \| \pa v_{\lambda_2}^{y_0}\|^2_{L_\omega^2}\big)r^{3+\delta_0}\mathrm{d}t \mathrm{d}r\nonumber\\
  \lesssim&~ \nu^{2+2\delta_0}\min\{\lambda_1,\lambda_2\} E(  w_{\lambda_1,\nu})E(v_{\lambda_2}).
\end{align}
Here we placed the $L^\infty_\omega$ norm on the factor $w_{\lambda_1,\nu}$, using $\|\pa w_{\lambda_1,\nu}\|_{L^\infty_\omega}\lesssim\nu\|\pa w_{\lambda_1,\nu}\|_{L^2_\omega}$.

Combining the above estimates, we complete the proof of \eqref{4.31}. 

The proof of  \eqref{4.35} and \eqref{4.29} is similar.  We introduce the notation
\begin{align*}
  &~ G_{\lambda_1,\nu_1}\cdot \pa_x \pa  u_{\lambda_2,\nu_2} = G(u_{\lambda_1,\nu_1},\pa_x u_{\lambda_2,\nu_2}),\\
 &~ [R_\nu,G_{\lambda_1,\nu_1}]\cdot \pa_x u_{\lambda_2,\nu_2} =  h^{i\alpha\beta }[R_\nu,\pa_\beta u_{\lambda_1,\nu_1} ]\pa_i\pa_{\alpha } u_{\lambda_2,\nu_2}.
\end{align*}
If  $\nu_1\gtrsim \nu$, we discard the commutator structure and obtain
\begin{align}
    &~ \| r^{\frac{1-\delta_0}{2}}\langle r \rangle^{\delta_0} [R_\nu,G_{\lambda_1,\nu_1}]\pa_x \pa  u_{\lambda_{2},\nu_2}\|_{L^2_{t,x}}\nonumber\\
    \lesssim&~ \min\{\nu_1,\nu_2\}\big\|r^{\frac{3-\delta_0}{2}}\langle r \rangle^{\delta_0} \|G_{\lambda_1,\nu_1}\|_{L^2_\omega}\cdot  \|\pa_x \pa  u_{\lambda_{2},\nu_2}\|_{L^2_\omega}\big\|_{L^2_{t,r}}\nonumber\\
    \lesssim&~ \min\{\nu_1,\nu_2\}\lambda_2 \sup_{y_0}\big\|r^{\frac{3-\delta_0}{2}}\langle r \rangle^{\delta_0} \|G_{\lambda_1,\nu_1}\|_{L^2_\omega}\cdot  \| (\pa  u_{\lambda_{2},\nu_2})^{y_0}\|_{L^2_\omega}\big\|_{L^2_{t,r}}
\end{align}
where we use \eqref{re} and always place the $ \| \cdot  \|_{L^\infty_\omega}$ norm on the low angular frequency terms, which gives the factor $\min\{\nu_1,\nu_2\}$. Following the same procedure as in the proof of \eqref{4.31}, we obtain \eqref{4.35} for $\nu_1\gtrsim \nu$. To obtain the interpolated estimate in \eqref{4.29}, we modify the above procedure slightly, namely
\begin{align}
  &~ \| r^{\frac{1-\delta_0}{2}}\langle r \rangle^{\delta_0} [R_\nu,G_{\lambda_1,\nu_1}]\pa_x \pa  u_{\lambda_{2},\nu_2}\|_{L^2_{t,x}}\nonumber\\
  \lesssim&~ \big\|r^{\frac{3-\delta_0}{2}}\langle r \rangle^{\delta_0} \|G_{\lambda_1,\nu_1}\|_{L^\infty_\omega}\cdot  \|\pa_x \pa  u_{\lambda_{2},\nu_2}\|_{L^2_\omega}\big\|_{L^2_{t,r}}\nonumber\\
  \lesssim&~ \nu_1 \big\|r^{\frac{3-\delta_0}{2}}\langle r \rangle^{\delta_0} \|G_{\lambda_1,\nu_1}\|_{L^2_\omega}\cdot \|\pa_x \pa  u_{\lambda_{2},\nu_2}\|_{L^2_\omega}^{s} \|\pa_x \pa  u_{\lambda_{2}}\|_{L^2_\omega}^{1-s}\big\|_{L^2_{t,r}}\nonumber\\
  \lesssim&~ \nu_1 \lambda_2 \sup_{y_0}\big\|r^{\frac{3-\delta_0}{2}}\langle r \rangle^{\delta_0} \|G_{\lambda_1,\nu_1}\|_{L^2_\omega}\cdot  \| (\pa  u_{\lambda_{2},\nu_2})^{y_0}\|_{L^2_\omega}^s\| (\pa  u_{\lambda_{2}})^{y_0}\|^{1-s}_{L^2_\omega}\big\|_{L^2_{t,r}},
\end{align}
where we place the $L^\infty_{\omega}$ norm on $G_{\lambda_1,\nu_1}$ and use Proposition \ref{2.5}. The rest is parallel to \eqref{4.31}.

If $\nu_1\ll \nu$, we apply the commutator estimate \eqref{angular} to gain the additional factor $\nu^{-1}\nu_1$:
\begin{align}
  &~ \|[R_\nu, G_{\lambda_1,\nu_1}]\cdot\pa_x\pa u_{\lambda_2,\nu_2}\|_{L^2_{\omega}} \lesssim \nu^{-1} \|\pa_\omega G_{\lambda_1,\nu_1} \|_{L^\infty_\omega}\cdot\|\pa_x \pa u_{\lambda_2,\nu_2}\|_{L^2_\omega}\notag\\
  \lesssim&~ \nu^{-1}\nu_1^2 \| G_{\lambda_1,\nu_1} \|_{L^2_\omega}\cdot\|\pa_x \pa u_{\lambda_2,\nu_2}\|_{L^2_\omega}\notag\\
   \lesssim&~ \nu^{-1}\nu_1^2 \lambda_2 \sup_{y_0}\| G_{\lambda_1,\nu_1} \|_{L^2_\omega}\cdot\|(\pa u_{\lambda_2,\nu_2})^{y_0}\|_{L^2_\omega},
\end{align}
and we may also replace the bound with
\begin{equation}
  \nu^{-1}\nu_1^2 \lambda_2 \sup_{y_0}\| G_{\lambda_1,\nu_1} \|_{L^2_\omega}\cdot\|(\pa u_{\lambda_2,\nu_2})^{y_0}\|_{L^2_\omega}^{s}\|(\pa u_{\lambda_2})^{y_0}\|_{L^2_\omega}^{1-s}
\end{equation}
and the conclusion follows from the procedure of proving \eqref{4.31}.

Notice that by \eqref{c2} we have 
\begin{equation}
  \sup_{y_\alpha}\|r^{-\frac{1-\delta_0}{2}}\langle r \rangle^{-\delta_0}(\pa v_1^{y_0}+ r_{y_1}^{-1} v_1^{y_0}) \|_{L^2_{t,x}}\lesssim E(v_1)^{\frac{1}{2}},
\end{equation}
together with \eqref{4.31}, this yields \eqref{4.33}.

For \eqref{4.34}, since $L_{t,x}^1$ is translation invariant, we may assume that the factor with the lowest angular frequency has zero translation. For instance, if $\nu_2\le\nu_1$, then we have
\begin{align*}
  &~ \sup_{y_\alpha}\|(\pa v_1^{y_0}+ r_{y_1}^{-1}v_1^{y_0})\cdot G(w_{\lambda_1,\nu_1},(v_{\lambda_2,\nu_2})^{y_2})  \|_{L^1_{t,x}} \\
  =&~ \sup_{y_\alpha}\|(\pa v_1^{y_0}+ r_{y_1}^{-1}v_1^{y_0})\cdot G((w_{\lambda_1,\nu_1})^{y_2},v_{\lambda_2,\nu_2})  \|_{L^1_{t,x}}\\
  \lesssim&~ \sup_{y_\alpha}\|r^{-\frac{1-\delta_0}{2}}\langle r \rangle^{-\delta_0}(\pa v_1^{y_0}+ r_{y_1}^{-1} v_1^{y_0}) \|_{L^2_{t,x}}\cdot \|r^{\frac{1-\delta_0}{2}}\langle r \rangle^{\delta_0}G((w_{\lambda_1,\nu_1})^{y_2},v_{\lambda_2,\nu_2})  \|_{L^2_{t,x}}\\
  \lesssim&~ \nu_{2}^{1+\delta_0}\min\{\lambda_1,\lambda_2\}^{\frac{1}{2}+\delta_0}E(v_1 )^{\frac{1}{2}}E( w_{\lambda_1,\nu_1})^{\frac{1}{2}}E(  v_{\lambda_2,\nu_2})^{\frac{1}{2}},
\end{align*}
and a similar argument applies when $\nu_1\le\nu_2$.

\end{proof}

\section{Frequency envelope formulation of bootstrap arguments}
In what follows we use the notion of frequency envelope, introduced by Tao \cite{tao_Global_2001}, which will be useful for tracking the evolution of
the energy of solutions between dyadic energy shells. A sequence $\{c_\lambda\}_{\lambda}\in l^2$ is said to be a frequency envelope if 
\begin{enumerate}
  \item It is slowly varying, namely
  $$\Big(\frac{\lambda}{\mu}\Big)^{\varepsilon}\lesssim \frac{c_\mu}{c_\lambda}\lesssim \Big(\frac{\mu}{\lambda}\Big)^{\varepsilon},\ \text{for}\ \mu\ge \lambda,$$
 where $\varepsilon$ is a small constant.
  \item It bounds the dyadic norms of $u$, namely $ \| u_\lambda  \|_{L^2}\le  C c_\lambda$.
\end{enumerate}
 
Next we place the initial data $\pa u[0] = (\pa_x u_0,u_1)$ in Theorem \ref{main} under suitable  frequency envelopes $\{c_\lambda\}_\lambda$ and $\{c_{\lambda,\nu}\}_{\nu}$. We define
\begin{equation}
\begin{aligned}
 &~ c_\lambda = \sup_{\lambda'} e^{-\delta^2_0 |\ln \lambda- \ln \lambda'|}M^{-1}\| P_{\lambda'} \pa u[0]\|_{H^{\frac{5}{2}+3\delta}},\\
 &~ c_{\lambda,\nu} =\nu^{-\delta_0^2}c_\lambda +  \sup_{\lambda',\nu'} e^{-\delta^2_0 (|\ln \lambda- \ln \lambda'|+|\ln \nu-\ln\nu'|)}N^{-1}\|R_{\nu'} P_{\lambda'} \Lambda_{\omega}^{2+\delta} \pa u[0]\|_{H^{\frac{1}{2}+\delta}} .
\end{aligned}
\end{equation}
It is clear that $\{c_\lambda\}_\lambda,\{c_{\lambda,\nu}\}_{\nu}$ are indeed frequency envelopes satisfying
$$\sum_\lambda c_\lambda^2,\ \sum_{\lambda,\nu} c_{\lambda,\nu}^2\lesssim 1,\ \nu^{-\delta_0^2 }c_\lambda\le  c_{\lambda,\nu},$$ 
and 
$$  \Big(\frac{\lambda}{\mu}\Big)^{\delta_0^2}\lesssim \frac{c_\mu}{c_\lambda},\frac{c_{\mu,\nu}}{c_{\lambda,\nu}}\lesssim \Big(\frac{\mu}{\lambda}\Big)^{\delta_0^2},\  \Big(\frac{\nu_2}{\nu_1}\Big)^{\delta_0^2}\lesssim \frac{c_{\lambda,\nu_1}}{c_{\lambda,\nu_2}}\lesssim \Big(\frac{\nu_1}{\nu_2}\Big)^{\delta_0^2},\ \text{for}\ \mu\ge  \lambda, \nu_1\ge \nu_2.$$
For the initial data we have
\begin{equation}
  \| P_\lambda \pa u[0] \|_{L^2}\le  M \lambda^{-\frac{5}{2}-3\delta}c_\lambda,\    \|R_\nu P_\lambda \pa u[0]  \|_{L^2}\le N\nu^{-2-\delta}\lambda^{-\frac{1}{2}-\delta}c_{\lambda,\nu}.
\end{equation}
By the boundedness of $R_\nu$, we can interpolate the above bounds as 
\begin{align}
 &~ \|R_\nu P_\lambda \pa u[0]  \|_{L^2}\lesssim  \|R_\nu P_\lambda \pa u[0]  \|_{L^2}^{\frac{1}{2}} \| P_\lambda \pa u[0] \|_{L^2}^{\frac{1}{2}} \nonumber\\
 \le&~ ( N \nu^{-2-\delta}\lambda^{-\frac{1}{2}-\delta}c_{\lambda,\nu})^{\frac{1}{2}} (M  \lambda^{-\frac{5}{2}-3\delta}c_\lambda)^{\frac{1}{2}}\nonumber\\
 \le&~ (MN)^{\frac{1}{2}} \lambda^{-\frac{3}{2}-2\delta}\nu^{-1-\frac{\delta}{2}}c_{\lambda,\nu}^{\frac{1}{2}} c_{\lambda}^{\frac{1}{2}}\nonumber \\
  \le&~ \varepsilon_0 \lambda^{-\frac{3}{2}-2\delta}\nu^{-1-\frac{\delta}{3}} c_{\lambda,\nu},\ \text{where }\varepsilon_0:=(MN)^{\frac{1}{2 }}\ll 1
\end{align}
and the last line follows from the fact that $\delta_0\ll\delta$ and $\nu^{-\delta_0^2}c_\lambda\le c_{\lambda,\nu}.$ 

Our goal is to bound the dyadic pieces of our solution with the same frequency envelopes. 
We decompose the solution $u$ as 
$$u=  \sum_{\lambda\in 2^{\mathbb{N}}}u_\lambda = \sum_{\lambda,\nu \in 2^{\mathbb{N}}}u_{\lambda,\nu}$$
and we have the equations for the dyadic pieces as
\begin{align}
  &~ \Box_{\bar{g}} u_\lambda =-[P_\lambda,\bar{h}^{\alpha \beta}\pa_{\alpha}\pa_{\beta}]u, \\
  &~ \Box_{\bar{g}} u_{\lambda,\nu}  =-[R_\nu P_\lambda,\bar{h}^{\alpha \beta}\pa_{\alpha}\pa_{\beta}]u,\\
  &~ \Box_{\bar{g}} \pa_t u_\lambda  =- [P_\lambda,\bar{h}^{\alpha \beta}\pa_{\alpha}\pa_{\beta}]\pa_t u - P_\lambda(\pa_t\bar{h}^{\alpha \beta}\pa_{\alpha}\pa_{\beta} u),\\
  &~ \Box_{\bar{g}} \pa_t u_{\lambda,\nu}  =-[R_\nu P_\lambda,\bar{h}^{\alpha \beta}\pa_{\alpha}\pa_{\beta}]\pa_t u -  R_\nu P_\lambda(\pa_t\bar{h}^{\alpha \beta}\pa_{\alpha}\pa_{\beta} u),
\end{align}
where $\bar{h}^{\alpha \beta} = -h^{\alpha\beta \gamma}\pa_\gamma u$.

We shall establish the following bounds for $u_\lambda, u_{\lambda,\nu}$, namely
\begin{theorem}\label{bound}
  Let $u$ be a solution to \eqref{quasi} in $[0,T]$ for arbitrarily large $T>0$, then we have energy bounds 
  \begin{equation}
    \|  \pa u_\lambda  \|_{L^2}\lesssim  M \lambda^{-\frac{5}{2}-3\delta}c_\lambda,\    \| \pa u_{\lambda,\nu}   \|_{L^2}\lesssim N \nu^{-2-\delta}\lambda^{-\frac{1}{2}-\delta}c_{\lambda,\nu}.
  \end{equation}
\end{theorem}

In view of the energy balance law \eqref{energy}, it is sufficient to control the functional $E(u_\lambda),E(u_{\lambda,\nu})$. To prove these bounds, we make the following bootstrap
assumptions, in which we assume the same bounds but with a worse constant $C$:
\begin{equation}\label{BA1}
  E(u_\lambda)^{\frac{1}{2}}, \lambda^{-1} E (\pa_t u_{\lambda})^{\frac{1}{2}}\le  CM  \lambda^{-\frac{5}{2}-3\delta}c_\lambda,
\end{equation}
and 
\begin{equation}\label{BA2}
  E (u_{\lambda,\nu})^{\frac{1}{2}}\le  CN \nu^{-2-\delta}\lambda^{-\frac{1}{2}-\delta}c_{\lambda,\nu}.
\end{equation}
Moreover, we assume the following refined bounds, obtained by interpolating the initial data bounds:
\begin{equation}\label{refined}
  E (u_{\lambda,\nu})^{\frac{1}{2}}\le  C\varepsilon_0 \lambda^{-\frac{3}{2}-2\delta}\nu^{-1-\frac{\delta}{3}} c_{\lambda,\nu} ,\ \lambda^{-1}E( \pa_t u_{\lambda,\nu})^{\frac{1}{2}}\le  C\varepsilon_0 \lambda^{-\frac{3}{2}-\delta}\nu^{-1-\frac{\delta}{3}} c_{\lambda,\nu} .
\end{equation}
Then we seek to improve the constants in these bounds. Since the initial data satisfy the bounds without the constant, it suffices to control the source terms. The gain comes from the fact that the space-time estimates provided by Proposition \ref{null} always carry an extra factor of $\varepsilon$. Once the bootstrap argument is closed, a continuity argument shows that
\begin{remark}
  It suffices to prove Theorems \ref{main} and \ref{bound} under the bootstrap assumptions \eqref{BA1}, \eqref{BA2}, and \eqref{refined}.
\end{remark}

\section{Proof of the global well-posedness results}
\subsection{Closing the bootstrap arguments}
Under the assumptions in Theorem \ref{main}, we have
\begin{equation}
  \varepsilon := (M^{1-\frac{\delta}{4}}N^{1+\frac{5\delta}{4}})^{\frac{1}{2+\delta}} + (M N)^{\frac{1}{2}}\ll 1.
\end{equation}
Using the refined bound \eqref{refined}, we verify the smallness condition \eqref{small_con} required by the local energy estimates. Indeed, we have 
\begin{equation}\label{sm}
  \begin{aligned}
    &~\|\langle r \rangle^{\delta_0}\bar{h}\|_{L^\infty} \\
    \lesssim&~\|\bar{h}\|_{L^\infty }  + \| \bar{h} \|^{1-\delta_0}_{L^\infty}\|r\bar{h}  \|^{\delta_0}_{L^\infty}\\
    \lesssim&~\|\bar{h}\|_{H^{\frac{3}{2}+\delta_0} }  + \| \bar{h} \|^{1-\delta_0}_{H^{\frac{3}{2}+\delta_0}}\|\Lambda_\omega^{1+\delta_0}\bar{h}  \|^{\delta_0}_{H^{\frac{1}{2}+\delta_0}}\\
    \lesssim &~ \|\Lambda_\omega^{1+\delta_0}\pa u\|_{H^{\frac{3}{2}+\delta_0}}\lesssim C\varepsilon \ll 1, \\
  \end{aligned}
\end{equation}
The first inequality is an interpolation between the $L^\infty$ and the weighted $L^\infty$ norms of $\bar{h}$; the second uses the Sobolev embedding in $x$, the angular Sobolev embedding, and the trace-type inequality \eqref{2.6}; the third follows from $\bar{h}=h\cdot\pa u$ and $\Lambda_\omega\ge 1$, and the last step uses the refined bound \eqref{refined} together with the summability of the envelopes.
We shall choose $C\varepsilon\ll 1$ so small that the coefficients $\bar{h} = h\cdot \pa u$ satisfy the smallness assumption in Proposition \ref{Le}.

Next, we interpolate the bootstrap bounds \eqref{BA2} and \eqref{refined} to obtain the refined bound with the smaller constant $\varepsilon_1$:
\begin{align}\label{refined2}
  &~ E (u_{\lambda,\nu})^{\frac{1}{2}} \le (C\varepsilon_0 \lambda^{-\frac{3}{2}-2\delta}\nu^{-1-\frac{\delta}{3}} c_{\lambda,\nu})^{\frac{4-\delta}{4+2\delta}}  (CN \nu^{-2-\delta}\lambda^{-\frac{1}{2}-\delta}c_{\lambda,\nu})^{1-\frac{4-\delta}{4+2\delta}}\nonumber\\
  \le&~ C \varepsilon_0^{\frac{4-\delta}{4+2\delta}} N^{ \frac{3\delta}{4+2\delta}} \lambda^{-\frac{3}{2}-\frac{6\delta}{5}}\nu^{-1-\delta}c_{\lambda,\nu}\nonumber\\
  =&~ C\varepsilon_1 \lambda^{-\frac{3}{2}-\frac{6\delta}{5}}\nu^{-1-\delta}c_{\lambda,\nu},\ \text{where}\ \varepsilon_1 := (M^{1-\frac{\delta}{4}}N^{1+\frac{5\delta}{4}})^{\frac{1}{2+\delta}}\ll 1.
\end{align}

To close the bootstrap argument, it remains to control the source terms; this is the content of the following proposition. 
\begin{proposition}\label{Source}Under the bootstrap assumptions, we have 
\begin{align}
  &~ \|\mathbf{F}(u^{y_0}_\lambda)\|_{L^1_{t,x}}\lesssim C\varepsilon( CM  \lambda^{-\frac{5}{2}-3\delta}c_\lambda)^2\label{u1},\\
  &~ \|\mathbf{F}((u_{\lambda,\nu})^{y_0})\|_{L^1_{t,x}}\lesssim  C \varepsilon (CN \nu^{-2-\delta}\lambda^{-\frac{1}{2}-\delta}c_{\lambda,\nu} )^2,  C  \varepsilon(C\varepsilon_0 \lambda^{-\frac{3}{2}-2\delta}\nu^{-1-\frac{\delta}{3}} c_{\lambda,\nu})^2,\label{u2}\\
  &~ \|\mathbf{F}((\pa_t u_\lambda)^{y_0})\|_{L^1_{t,x}}\lesssim C \varepsilon\lambda^{2}(C M \lambda^{-\frac{5}{2}-3\delta}c_\lambda)^2\label{u3},\\
  &~ \|\mathbf{F}((\pa_t u_{\lambda,\nu})^{y_0})\|_{L^1_{t,x}}\lesssim  C \varepsilon \lambda^2 (C \varepsilon_0 \lambda^{-\frac{3}{2}-\delta}\nu^{-1-\frac{\delta}{3}} c_{\lambda,\nu})^2 ,\label{u4}
\end{align}
where $\mathbf{F}(v) = \mathbf{F}_e(v)+ \mathbf{F}_m(v)$.
\end{proposition}
\begin{proof}Since $L_{t,x}^1$ is translation invariant, we may add translations to the source terms $\mathbf{F}_e(u^{y_0}_\lambda), \mathbf{F}_m(u^{y_0}_\lambda)$. We will repeatedly use the following summation fact for the high frequency pieces of the envelopes:
  \begin{align*}
  &~ \sum_{\lambda_1\gtrsim \lambda,\nu_1\gtrsim \nu} \lambda_1^{-s_1}\nu_1^{-s_2}c_{\lambda_1,\nu_1}
  \lesssim&~ \sum_{\lambda_1\gtrsim \lambda,\nu_1\gtrsim \nu}\lambda_1^{-s_1}\nu_1^{-s_2} \Big(\frac{\lambda_1 \nu_1}{\lambda \nu}\Big)^{\delta_0^2}c_{\lambda,\nu}
  \lesssim&~ \lambda^{-s_1}\nu^{-s_2}c_{\lambda,\nu},\ \text{for}\ s_1,s_2\ge  2\delta_0^2.
   \end{align*}

\textbf{(A). Estimates for \eqref{u1}.} 
Up to a translation, the source terms in \eqref{u1} take the form 
 $$\begin{aligned}
 &~ |\mathbf{F}_e(u^{y_0}_\lambda)|,|\mathbf{F}_m(u^{y_0}_\lambda)| \sim  |(\pa u_\lambda + r_{y_0}^{-1} u_\lambda)\cdot [P_\lambda,\bar{h}^{\alpha \beta} ]\pa_\alpha \pa_\beta u|+|(\pa u_\lambda + r_{y_0}^{-1} u_\lambda)\cdot G(\pa u,u_\lambda)| .
  \end{aligned} $$
The first term comes from the commutator $[P_\lambda,\bar{h}^{\alpha \beta}\pa_\alpha\pa_\beta]$, and the second from the null form structure of the source; see \eqref{r1}.

  Using Lemma \ref{pcom} and the $L$-notation, we expand the commutator structure as 
\begin{align}
   [P_\lambda,\bar{h}^{\alpha \beta} \pa_\alpha \pa_\beta ]u
 =  \sum_{\substack{\lambda_1\gtrsim \lambda, \lambda_2}}h^{i\alpha\beta}L(\pa_\beta u_{\lambda_1},  \pa_i \pa_\alpha u_{\lambda_2})+ \sum_{\lambda_1 \ll \lambda_2\sim \lambda}\lambda^{-1}\lambda_1 h^{i\alpha\beta}L( \pa_\beta u_{\lambda_1},  \pa_i \pa_\alpha u_{\lambda_2})
\end{align}
In the first summation, where $\lambda_1\gtrsim\lambda$, no gain is available from the commutator; in the second, where $\lambda_1\ll\lambda$, Lemma \ref{pcom} yields the factor $\lambda^{-1}\lambda_1$.
In view of \eqref{re} and \eqref{r1}, we have 
\begin{align}
  &~ \|(\pa u_\lambda + r_{y_0}^{-1} u_\lambda) [P_\lambda,\bar{h}^{\alpha \beta} ]\pa_\alpha \pa_\beta u\|_{L^1_{t,x}}\notag\\
  \lesssim&~ \sup_{y_\alpha}\sum_{\lambda_{1}\gtrsim \lambda,\lambda_2}\sum_\nu \lambda_2 \|(\pa u^{y_1}_\lambda+r^{-1}_{y_0}u^{y_1}_\lambda) G( u_{\lambda_2,\nu}, u^{y_2}_{\lambda_1})  \|_{L^1_{t,x}}\notag\\
  &~ \quad + \sup_{y_\alpha}\sum_{\lambda_{1}\ll \lambda_2 \sim \lambda}\sum_\nu \lambda^{-1}\lambda_1 \lambda_2\|(\pa u^{y_1}_\lambda+r^{-1}_{y_0}u^{y_1}_\lambda) G( u^{y_2}_{\lambda_2}, u_{\lambda_1,\nu})  \|_{L^1_{t,x}}.
\end{align}
Applying \eqref{4.31} in Proposition \ref{null}, the above summation is bounded by 
\begin{align}
  &~ E ( u_\lambda)^{\frac{1}{2}}\sum_{\lambda_{1}\gtrsim \lambda,\lambda_2}\sum_\nu\nu^{1+\delta_0}\lambda_{2}^{\frac{3}{2}+\delta_0}E( u_{\lambda_2,\nu})^{\frac{1}{2}}E(u_{\lambda_1})^{\frac{1}{2}}\nonumber\\
  &+  E ( u_\lambda)^{\frac{1}{2}}\sum_{\lambda_{1}\ll \lambda_2 \sim \lambda}\sum_\nu\nu^{1+\delta_0}\lambda^{-1}\lambda_{1}^{\frac{3}{2}+\delta_0} \lambda_2 E( u_{\lambda_2})^{\frac{1}{2}}E(u_{\lambda_1,\nu})^{\frac{1}{2}}\notag\\
  \lesssim&~ C\varepsilon_0( E ( u_\lambda) +  \sum_{\lambda_2 \sim \lambda }E( u_\lambda)^{\frac{1}{2}}  \lambda^{-1}\lambda_2 E( u_{\lambda_2})^{\frac{1}{2}})
  \lesssim  C\varepsilon_0( CM  \lambda^{-\frac{5}{2}-3\delta}c_\lambda)^2,
\end{align}
where we applied the refined bound \eqref{refined} to $E( u_{\lambda_1,\nu})$ and $E(u_{\lambda_2,\nu})$.

The term $(\pa u_\lambda + r_{y_0}^{-1} u_\lambda)\cdot G(\pa_x u,u_\lambda)$ is estimated in the same way. For the term with $G(\pa_t u,u_\lambda)$, we apply \eqref{refined} to $E(\pa_t u_{\lambda,\nu})$ and obtain 
\begin{align}
  &~ \|(\pa u_\lambda + r_{y_0}^{-1} u_\lambda)\cdot G(\pa_t u,u_\lambda)\|_{L^1_{t,x}}\notag\\
  \lesssim&~ \sum_{\mu, \nu} \|(\pa u_\lambda + r_{y_0}^{-1} u_\lambda)\cdot G(\pa_t u_{\mu,\nu},u_\lambda)\|_{L^1_{t,x}}\notag\\
  \lesssim&~ E(u_\lambda) \sum_{\mu,\nu}\nu^{1+\delta_0}\mu^{\frac{1}{2}+\delta_0} E(\pa_t u_{\mu,\nu})^{\frac{1}{2}}
  \lesssim  C\varepsilon_0 E(u_\lambda).
\end{align}
\textbf{(B). Estimates for \eqref{u2}.}
We again write the source terms up to a translation as
\begin{equation}
  \begin{aligned}
  &~ |\mathbf{F}_e((u_{\lambda,\nu})^{y_0})|,|\mathbf{F}_m((u_{\lambda,\nu})^{y_0})| \\
 \sim  &~ |(\pa u_{\lambda,\nu} + r_{y_0}^{-1} u_{\lambda,\nu})\cdot [R_\nu P_\lambda,\bar{h}^{\alpha \beta} \pa_\alpha \pa_\beta] u|+ |(\pa u_{\lambda,\nu} + r_{y_0}^{-1} u_{\lambda,\nu})\cdot G(\pa u,u_{\lambda,\nu})| .
   \end{aligned}
\end{equation}

As before, we first estimate the terms without commutator structure: 
\begin{align}
  &~ \|(\pa u_{\lambda,\nu} + r_{y_0}^{-1} u_{\lambda,\nu})\cdot G(\pa u,u_{\lambda,\nu})\|_{L^1_{t,x}}\notag\\
  \lesssim&~ E(u_{\lambda,\nu})\sum_{\mu,\nu}\mu^{\frac{3}{2}+\delta_0} \nu^{1+\delta_0} E( u_{\mu,\nu})^{\frac{1}{2}} +  E(u_{\lambda,\nu})\sum_{\mu,\nu}\mu^{\frac{1}{2}+\delta_0} \nu^{1+\delta_0} E(\pa_t  u_{\mu,\nu})^{\frac{1}{2}}\notag\\
  \lesssim&~ C\varepsilon_0 E(u_{\lambda,\nu}).
\end{align}

To proceed, we decompose the commutator as follows: 
\begin{align}
  &~ [R_\nu P_\lambda,\bar{h}^{\alpha\beta} \pa_\alpha \pa_\beta] u = P_\lambda([R_\nu ,\bar{h}^{\alpha\beta} \pa_\alpha \pa_\beta] u)+ [P_\lambda, \bar{h}^{\alpha \beta}\pa_\alpha \pa_\beta]R_\nu u.
\end{align}

For the terms involving the $P_\lambda$ commutator structure, we have the following expansion:
\begin{align}
    [P_\lambda, \bar{h}^{\alpha \beta}\pa_\alpha \pa_\beta]R_\nu u = &\sum_{\lambda_1\gtrsim \lambda,\lambda_2}\sum_{\nu'}h^{i\alpha\beta }L(\pa_\beta u_{\lambda_1,\nu'},  \pa_i\pa_\alpha  u_{\lambda_2,\nu})\notag\\
  &  + \sum_{\lambda_1\ll \lambda_2\sim \lambda }\sum_{\nu'}  \lambda^{-1}\lambda_1 h^{i\alpha\beta }L(\pa_\beta u_{\lambda_1,\nu'},  \pa_i\pa_\alpha  u_{\lambda_2,\nu}).
\end{align}

Repeating the procedure of proving \eqref{u1}, we have
\begin{align}\label{S1}
  &~ \|(\pa u_{\lambda,\nu} + r_{y_0}^{-1} u_{\lambda,\nu})\cdot [P_\lambda, \bar{h}^{\alpha \beta}\pa_\alpha \pa_\beta]R_\nu u\|_{L^{1}_{t,x}}\nonumber \\
  \lesssim&~ E(u_{\lambda,\nu})^{\frac{1}{2}}\sum_{\lambda_1 \gtrsim \lambda,\lambda_2}\sum_{\nu'} \min\{\nu',\nu\}^{1+\delta_0} \lambda_{2}^{\frac{3}{2}+\delta_0}  E(u_{\lambda_1,\nu'})^{\frac{1}{2}}E(u_{\lambda_2,\nu})^{\frac{1}{2}} \nonumber \\
  &   + E(u_{\lambda,\nu})^{\frac{1}{2}}\sum_{\lambda_1 \ll\lambda_2\sim \lambda }\sum_{\nu'}\min\{\nu',\nu\}^{1+\delta_0}\lambda^{-1}  \lambda_{1}^{\frac{3}{2}+\delta_0} \lambda_2E(u_{\lambda_1,\nu'})^{\frac{1}{2}}E(u_{\lambda_2,\nu})^{\frac{1}{2}}  .
\end{align}
For the first summation in \eqref{S1}, we apply the refined bound \eqref{refined} to $E(u_{\lambda_1,\nu'})$ and  control the summation by
\begin{align}
  &~ E(u_{\lambda,\nu})^{\frac{1}{2}}\sum_{\lambda_1 \gtrsim \lambda,\lambda_2} \sum_{\nu'}\nu'^{1+\delta_0} \lambda_{2}^{\frac{3}{2}+\delta_0} \lambda_2 E(u_{\lambda_1,\nu'})^{\frac{1}{2}}E(u_{\lambda_2,\nu})^{\frac{1}{2}} \nonumber \\
  \lesssim&~ E(u_{\lambda,\nu})^{\frac{1}{2}}\sum_{\lambda_1 \gtrsim \lambda,\lambda_2}\sum_{\nu'}  \nu'^{1+\delta_0}\lambda_{2}^{\frac{3}{2}+\delta_0}  (C\varepsilon_0 \lambda_1^{-\frac{3}{2}-2\delta}\nu'^{-1-\frac{\delta}{3}} c_{\lambda_1,\nu'}) E(u_{\lambda_2,\nu})^{\frac{1}{2}} \nonumber \\
  \lesssim&~ C\varepsilon_0 E(u_{\lambda,\nu})^{\frac{1}{2}}\sum_{\lambda_1 \gtrsim \lambda,\lambda_2}\sum_{\nu'} \nu'^{\delta_0-\frac{\delta}{3}} \lambda_{2}^{\frac{3}{2}+\delta_0}\lambda_1^{-\frac{3}{2}-2\delta}c_{\lambda_1,\nu'}   E(u_{\lambda_2,\nu})^{\frac{1}{2}}.
\end{align}
The remaining summations converge since $\delta_0\ll\delta$ and the envelopes are slowly varying.
Next we apply \eqref{BA2} and \eqref{refined2} to $E(u_{\lambda_2,\nu})$ respectively to obtain
\begin{align}
  &~ \sum_{\lambda_1 \gtrsim \lambda,\lambda_2}  \lambda_{2}^{\frac{3}{2}+\delta_0} \lambda_1^{-\frac{3}{2}-2\delta} E(u_{\lambda_2,\nu})^{\frac{1}{2}}\lesssim \sum_{\lambda_1 \gtrsim \lambda,\lambda_2}  \lambda_{2}^{\frac{3}{2}+\delta_0}  \lambda_1^{-\frac{3}{2}-2\delta}(CN \nu^{-2-\delta}\lambda_2^{-\frac{1}{2}-\delta}c_{\lambda_2,\nu} ) \nonumber\\
  \lesssim&~ CN \nu^{-2-\delta}c_{\lambda,\nu}\Big(\sum_{\lambda_1 \gtrsim \lambda\gtrsim\lambda_2}  \lambda_{2}^{1+\delta_0-\delta}  \lambda_1^{-\frac{3}{2}-2\delta} \big(\frac{\lambda}{\lambda_2}\big)^{\delta_0^2} + \sum_{\lambda_1 \gtrsim \lambda_2\gtrsim\lambda} \lambda_{2}^{1+\delta_0-\delta}  \lambda_1^{-\frac{3}{2}-2\delta} \big(\frac{\lambda_2}{\lambda}\big)^{\delta_0^2}\Big) \nonumber\\
  \lesssim&~ CN \nu^{-2-\delta} \lambda^{-\frac{1}{2}-\delta}c_{\lambda,\nu}
\end{align}
and 
\begin{align}
  &~ C\varepsilon_0\sum_{\lambda_1 \gtrsim \lambda,\lambda_2}\sum_{\nu'} \nu'^{\delta_0-\delta} \lambda_{2}^{\frac{3}{2}+\delta_0}\lambda_1^{-\frac{3}{2}-2\delta}c_{\lambda_1,\nu'}   E(u_{\lambda_2,\nu})^{\frac{1}{2}}\nonumber\\
  \lesssim&~ C\varepsilon_0\sum_{\lambda_1 \gtrsim \lambda,\lambda_2}\sum_{\nu'} \nu'^{\delta_0-\delta} \lambda_{2}^{\frac{3}{2}+\delta_0}\lambda_1^{-\frac{3}{2}-2\delta}c_{\lambda_1,\nu'}  (C\varepsilon_1 \lambda_2^{-\frac{3}{2}-\frac{6\delta}{5}} \nu^{-1-\delta}c_{\lambda_2,\nu}) \nonumber\\
  \lesssim&~ C^2\varepsilon_0\varepsilon_1\sum_{\lambda_1 \gtrsim \lambda,\lambda_2}\lambda_{2}^{\delta_0-\frac{6\delta}{5}}\lambda_1^{-\frac{3}{2}-2\delta}c_{\lambda_2,\nu}\Big(\sum_{\nu'\gtrsim \nu} \nu'^{\delta_0-\delta} \nu^{-1-\delta}c_{\lambda_1,\nu'} +\sum_{\nu'\ll \nu} \nu'^{\delta_0-\delta} \nu^{-1-\delta}\big(\frac{\nu}{\nu'}\big)^{\delta_0^2}c_{\lambda_1,\nu} \Big)   \nonumber\\
  \lesssim&~ C^2\varepsilon_0\varepsilon_1 \nu^{-1-\frac{\delta}{3}}\sum_{\lambda_1 \gtrsim \lambda,\lambda_2}\lambda_2^{\delta_0-\frac{6\delta}{5}}\lambda_1^{-\frac{3}{2}-2\delta}c_{\lambda_1,\nu} c_{\lambda_2,\nu}\nonumber\\
  \lesssim&~ C\varepsilon_1\cdot C\varepsilon_0\lambda^{-\frac{3}{2}-2\delta}\nu^{-1-\frac{\delta}{3}} c_{\lambda,\nu},
\end{align}
Both bounds are acceptable.

For the second summation in \eqref{S1},  we apply the refined bound \eqref{refined} to $E(u_{\lambda_1,\nu'})$ to obtain
\begin{align*}
  &~ \sum_{\lambda_1 \ll\lambda\sim \lambda_2}\sum_{\nu'} \nu'^{1+\delta_0} \lambda^{-1} \lambda_{1 }^{\frac{3}{2}+\delta_0} \lambda_2E(u_{\lambda,\nu})^{\frac{1}{2}}E(u_{\lambda_1,\nu'})^{\frac{1}{2}}E(u_{\lambda_2,\nu})^{\frac{1}{2}} \\
  \lesssim&~ E(u_{\lambda,\nu})^{\frac{1}{2}}\sum_{\lambda_1 \ll\lambda\sim \lambda_2}\sum_{\nu'}  \nu'^{1+\delta_0} \lambda_{1}^{\frac{3}{2}+\delta_0} E(u_{\lambda_1,\nu'})^{\frac{1}{2}}E(u_{\lambda_2,\nu})^{\frac{1}{2}}
  \lesssim  C \varepsilon_0 E(u_{\lambda,\nu})^{\frac{1}{2}}\sum_{\lambda\sim \lambda_2}  E(u_{\lambda_2,\nu})^{\frac{1}{2}},
  \end{align*}
which is acceptable.

For the terms involving the $R_\nu$-commutator structure, we apply \eqref{4.35} and obtain
\begin{align}\label{S2}
  &~ \sum_{\lambda_{\text{max}}\sim \lambda_{\text{med}}}\sum_{\nu_1,\nu_2}\|(\pa u_{\lambda,\nu} + r_{y_0}^{-1} u_{\lambda,\nu})\cdot P_\lambda([R_\nu ,h^{i \alpha\beta}\pa_\beta u_{\lambda_1,\nu_1} \pa_i\pa_\alpha ] u_{\lambda_2,\nu_2})\|_{L^1_{t,x}}\notag\\
  \lesssim&~ \sum_{\lambda_{\text{max}}\sim \lambda_{\text{med}}}\sum_{\nu_1,\nu_2}\sup_{y_\alpha}\|((\pa u_{\lambda,\nu})^{y_1} + r_{y_0}^{-1} (u_{\lambda,\nu})^{y_1})\cdot [R_\nu ,h^{i\alpha\beta}\pa_\beta u_{\lambda_1,\nu_1} \pa_i \pa_\alpha ] u_{\lambda_2,\nu_2}\|_{L^1_{t,x}}\notag\\
  \lesssim&~ \sum_{\lambda_{\text{max}}\sim \lambda_{\text{med}}}\sum_{\nu_1\gtrsim \nu,\nu_2}E(u_{\lambda,\nu})^{\frac{1}{2}}\nu_{2}^{1+\delta_0}\lambda_{\text{min}}^{\frac{1}{2}+\delta_0}\lambda_2 E( u_{\lambda_1,\nu_1})^{\frac{1}{2}}E( u_{\lambda_2,\nu_2})^{\frac{1}{2}}\notag\\
  &~  + \sum_{\lambda_{\text{max}}\sim \lambda_{\text{med}}}\sum_{\nu_1\ll\nu_2\sim \nu}E(u_{\lambda,\nu})^{\frac{1}{2}}\nu^{-1}\nu_{1}^{2+\delta_0}\lambda_{\text{min}}^{\frac{1}{2}+\delta_0}\lambda_2 E( u_{\lambda_1,\nu_1})^{\frac{1}{2}}E( u_{\lambda_2,\nu_2})^{\frac{1}{2}},
\end{align}
where $\lambda_{\text{med}}$ is  the median of $\lambda,\lambda_1,\lambda_2$. 
Here we applied \eqref{4.35} together with the $L^1_{t,x}$ estimate \eqref{4.34}, and split the angular summation according to whether $\nu_1\gtrsim\nu$ or $\nu_1\ll\nu_2\sim\nu$.

We first consider the summation in \eqref{S2} involving the bound $CN \nu^{-2-\delta}\lambda^{-\frac{1}{2}-\delta}c_{\lambda,\nu}$. Applying \eqref{BA2} to $E(u_{\lambda_1,\nu_1})$ and \eqref{refined2} to $E( u_{\lambda_2,\nu_2})$, we bound the summation by
\begin{align}
  &~ E(u_{\lambda,\nu})^{\frac{1}{2}}C^2 \varepsilon_1 N \sum_{\lambda_{\text{max}}\sim \lambda_{\text{med}}}\sum_{\nu_1\gtrsim \nu,\nu_2}\nu_1^{-2-\delta}\nu_2^{\delta_0-\delta}\lambda_{\text{min}}^{\frac{1}{2}+\delta_0}\lambda_1^{-\frac{1}{2}-\delta}\lambda_2^{-\frac{1}{2}-\frac{6\delta}{5}}c_{\lambda_1,\nu_1}c_{\lambda_2,\nu_2}\notag\\
 + & ~ E(u_{\lambda,\nu})^{\frac{1}{2}}C^2 \varepsilon_1 N \sum_{\lambda_{\text{max}}\sim \lambda_{\text{med}}}\sum_{\nu_1\ll\nu_2\sim \nu} \nu^{-1}\nu_1^{\delta_0-\delta}\nu_2^{-1-\delta}\lambda_{\text{min}}^{\frac{1}{2}+\delta_0}\lambda_1^{-\frac{1}{2}-\delta}\lambda_2^{-\frac{1}{2}-\frac{6\delta}{5}}c_{\lambda_1,\nu_1}c_{\lambda_2,\nu_2}.
\end{align} 
In the case of $\lambda_2\gtrsim \lambda,\lambda_1$, we have
\begin{align}
  &~ \sum_{\lambda_2\gtrsim \lambda,\lambda_1}\sum_{\nu_1\gtrsim \nu,\nu_2} \nu_1^{-2-\delta}\nu_2^{\delta_0-\delta}\lambda_{\text{min}}^{\frac{1}{2}+\delta_0}\lambda_1^{-\frac{1}{2}-\delta}\lambda_2^{-\frac{1}{2}-\frac{6\delta}{5}}c_{\lambda_1,\nu_1}c_{\lambda_2,\nu_2}\notag\\
  & + \sum_{\lambda_2\gtrsim \lambda,\lambda_1}\sum_{\nu_1\ll\nu_2\sim \nu}\nu^{-1}\nu_1^{\delta_0-\delta}\nu_2^{-1-\delta}\lambda_{\text{min}}^{\frac{1}{2}+\delta_0}\lambda_1^{-\frac{1}{2}-\delta}\lambda_2^{-\frac{1}{2}-\frac{6\delta}{5}}c_{\lambda_1,\nu_1}c_{\lambda_2,\nu_2}.\notag\\
  \lesssim&~ \nu^{-2-\delta} \sum_{\lambda_2\gtrsim \lambda,\lambda_1}\lambda_1^{\delta_0-\delta}\lambda_2^{-\frac{1}{2}-\frac{6\delta}{5}}\Big(\sum_{\nu_1\gtrsim \nu,\nu_2}\big(\frac{\nu_1}{\nu}\big)^{-2-\delta} \nu_2^{\delta_0-\delta}\big(\frac{\lambda_2}{\lambda_1}\big)^{\delta_0^2}c_{\lambda_1,\nu_1}c_{\lambda_2,\nu_2} + \sum_{\nu_1\ll\nu_2\sim \nu}\nu_1^{\delta_0-\delta}c_{\lambda_1,\nu_1}c_{\lambda_2,\nu_2}\Big)\nonumber\\
  \lesssim&~ \nu^{-2-\delta}\lambda^{-\frac{1}{2}-\delta}c_{\lambda,\nu},
\end{align}
which is acceptable. 

In the case of $\lambda_2\ll \lambda_1\sim  \lambda$ and $\nu_1\gtrsim \nu,\nu_2$, we have
\begin{align}
  &~ \sum_{\lambda_2\ll \lambda_1\sim  \lambda}\sum_{\nu_1\gtrsim \nu,\nu_2}\nu_{\text{min}}^{1+\delta_0}\nu_1^{-2-\delta}\nu_2^{-1-\delta}\lambda_{\text{min}}^{\frac{1}{2}+\delta_0}\lambda_1^{-\frac{1}{2}-\delta}\lambda_2^{-\frac{1}{2}-\frac{6\delta}{5}}c_{\lambda_1,\nu_1}c_{\lambda_2,\nu_2}\notag\\
  \lesssim&~ \nu^{-2-\delta} \lambda^{-\frac{1}{2}-\delta}\sum_{\lambda_2\ll \lambda_1\sim  \lambda}\sum_{\nu_1\gtrsim \nu,\nu_2}\lambda_2^{\delta_0 - \frac{6\delta}{5}}\big(\frac{\nu_1}{\nu}\big)^{-2-\delta}\nu_2^{\delta_0-\delta}c_{\lambda_1,\nu_1}c_{\lambda_2,\nu_2}\nonumber\\
  \lesssim&~ \nu^{-2-\delta}\lambda^{-\frac{1}{2}-\delta}c_{\lambda,\nu},
\end{align}
which is also acceptable.

The remaining case is $\lambda_2\ll \lambda_1\sim  \lambda$ and $\nu_1\ll \nu_2\sim \nu$. In this case we again interpolate the bounds \eqref{BA2} and \eqref{refined} at the scale $(\lambda_1,\nu_1)$ with the weight $\theta=\frac{3\delta}{4+2\delta}$ to obtain
\begin{equation}
  \begin{aligned}
  &~ E(u_{\lambda_1,\nu_1})^{\frac{1}{2}} \le (C\varepsilon_0 \lambda_1^{-\frac{3}{2}-2\delta}\nu_1^{-1-\frac{\delta}{3}} c_{\lambda_1,\nu_1})^{\frac{3\delta}{4+2\delta}}  (CN \nu_1^{-2-\delta}\lambda_1^{-\frac{1}{2}-\delta}c_{\lambda_1,\nu_1})^{1-\frac{3\delta}{4+2\delta}} \\
  \lesssim&~ C \big(\frac{M}{N}\big)^{\frac{3\delta}{4(2+\delta)}} N \nu_1^{-2-\frac{\delta}{5}}\lambda_1^{-\frac{1}{2}-\frac{5\delta}{4}}c_{\lambda_1,\nu_1}.
   \end{aligned} 
\end{equation}

 Applying the above bound to $E(u_{\lambda_1,\nu_1})$ and \eqref{refined2} to $E(u_{\lambda_2,\nu_2})$, we have 
\begin{align}
  &~ \sum_{\lambda_2\ll \lambda_1\sim  \lambda}\sum_{\nu_1\ll \nu_2\sim \nu}E(u_{\lambda,\nu})^{\frac{1}{2}}\nu_{\text{min}}^{1+\delta_0}\lambda_{\text{min}}^{\frac{1}{2}+\delta_0}\lambda_2 E( u_{\lambda_1,\nu_1})^{\frac{1}{2}}E( u_{\lambda_2,\nu_2})^{\frac{1}{2}}\nonumber\\
  \lesssim&~ E(u_{\lambda,\nu})^{\frac{1}{2}}C^2 [\big(\frac{M}{N}\big)^{\frac{3\delta}{4(2+\delta)}} \varepsilon_1] N \sum_{\lambda_2\ll \lambda_1\sim  \lambda}\sum_{\nu_1\ll \nu_2\sim \nu}\nu_{\text{min}}^{1+\delta_0}\nu^{-1}\nu_1^{-1-\frac{\delta}{5}}\nu_2^{-1-\delta}\lambda_{\text{min}}^{\frac{1}{2}+\delta_0}\lambda_1^{-\frac{1}{2}-\frac{5\delta}{4}}\lambda_2^{-\frac{1}{2}-\frac{6\delta}{5}}c_{\lambda_1,\nu_1}c_{\lambda_2,\nu_2}\notag\\
  \lesssim&~ E(u_{\lambda,\nu})^{\frac{1}{2}}C^2 \varepsilon_0 N \sum_{\lambda_2\ll \lambda_1\sim  \lambda}\sum_{\nu_1\ll \nu_2\sim \nu}\nu^{-1}\nu_1^{\delta_0-\frac{\delta}{5}}\nu_2^{-1-\delta}\lambda_1^{-\frac{1}{2}-\frac{5\delta}{4}}\lambda_2^{\delta_0-\frac{6\delta}{5}}c_{\lambda_1,\nu_1}\cdot \big(\frac{\lambda}{\lambda_2}\big)^{\delta_0^2}c_{\lambda,\nu_2}\notag\\
  \lesssim&~ E(u_{\lambda,\nu})^{\frac{1}{2}} C\varepsilon_0\cdot  CN\lambda^{-\frac{1}{2}-\delta}\nu^{-1}\sum_{\nu_2\sim  \nu}\nu_2^{-1-\delta}c_{\lambda,\nu_2}\nonumber\\
  \lesssim&~ C \varepsilon_0 E(u_{\lambda,\nu})^{\frac{1}{2}} \cdot C N \lambda^{-\frac{1}{2}-\delta}\nu^{-2-\delta}c_{\lambda,\nu},
\end{align}
which is acceptable.

It remains to control the summation in \eqref{S2} with the bound $C\varepsilon_0 \lambda^{-\frac{3}{2}-2\delta}\nu^{-1-\frac{\delta}{3}}$.

We begin with the case $\lambda_1\gtrsim \lambda,\lambda_2$, where we apply \eqref{refined} to  $E(u_{\lambda_1,\nu_1})$ and \eqref{refined2} to $E(u_{\lambda_2,\nu_2})$, and obtain
\begin{align}
  &~ E(u_{\lambda,\nu})^{\frac{1}{2}}\sum_{\lambda_1\gtrsim \lambda,\lambda_2}\sum_{\nu_1\gtrsim \nu,\nu_2}\nu_{2}^{1+\delta_0}\lambda_{\text{min}}^{\frac{1}{2}+\delta_0}\lambda_2 E( u_{\lambda_1,\nu_1})^{\frac{1}{2}}E( u_{\lambda_2,\nu_2})^{\frac{1}{2}}\notag\\
  & + E(u_{\lambda,\nu})^{\frac{1}{2}}\sum_{\lambda_1\gtrsim \lambda,\lambda_2}\sum_{\nu_1\ll\nu_2\sim \nu}\nu^{-1} \nu_{1}^{2+\delta_0}\lambda_{\text{min}}^{\frac{1}{2}+\delta_0}\lambda_2 E( u_{\lambda_1,\nu_1})^{\frac{1}{2}}E( u_{\lambda_2,\nu_2})^{\frac{1}{2}},\nonumber\\
  \lesssim&~ E(u_{\lambda,\nu})^{\frac{1}{2}}\sum_{\lambda_1\gtrsim \lambda,\lambda_2}\sum_{\nu_1\gtrsim \nu,\nu_2}\nu_{2}^{1+\delta_0}\lambda_{2}^{\frac{3}{2}+\delta_0} (C\varepsilon_0 \lambda_1^{-\frac{3}{2}-2\delta}\nu_1^{-1-\frac{\delta}{3}} c_{\lambda_1,\nu_1})(C\varepsilon_1 \lambda_2^{-\frac{3}{2}-\frac{6\delta}{5}}\nu_2^{-1-\delta}c_{\lambda_2,\nu_2})\notag\\
  & + E(u_{\lambda,\nu})^{\frac{1}{2}}\sum_{\lambda_1\gtrsim \lambda,\lambda_2}\sum_{\nu_1\ll\nu_2\sim \nu}\nu^{-1} \nu_{1}^{2+\delta_0}\lambda_{2}^{\frac{3}{2}+\delta_0} (C\varepsilon_0 \lambda_1^{-\frac{3}{2}-2\delta}\nu_1^{-1-\frac{\delta}{3}} c_{\lambda_1,\nu_1})(C\varepsilon_1 \lambda_2^{-\frac{3}{2}-\frac{6\delta}{5}}\nu_2^{-1-\delta}c_{\lambda_2,\nu_2}),\nonumber\\
  \lesssim&~ E(u_{\lambda,\nu})^{\frac{1}{2}}C^2 \varepsilon_0\varepsilon_1\sum_{\lambda_1\gtrsim \lambda,\lambda_2}\sum_{\nu_1\gtrsim \nu,\nu_2}\nu_{2}^{\delta_0-\delta}\lambda_2^{\delta_0-\frac{6\delta}{5}}\lambda_1^{-\frac{3}{2}-2\delta} \nu_1^{-1-\frac{\delta}{3}} c_{\lambda_1,\nu_1} c_{\lambda_2,\nu_2}\notag\\
  & + C^{2}\varepsilon_0\varepsilon_1\sum_{\lambda_1\gtrsim \lambda,\lambda_2}\sum_{\nu_1\ll\nu_2\sim \nu}\nu^{-2-\delta} \nu_{1}^{1+\delta_0-\frac{\delta}{3}}\lambda_{2}^{\delta_0-\frac{6\delta}{5}}\lambda_2  \lambda_1^{-\frac{3}{2}-2\delta} \big(\frac{\nu}{\nu_1}\big)^{\delta_0^2} c_{\lambda_1,\nu} c_{\lambda_2,\nu_2} ,\nonumber\\
  \lesssim&~ C\varepsilon_1E(u_{\lambda,\nu})^{\frac{1}{2}}\cdot C\varepsilon_0 \lambda^{-\frac{3}{2}-2\delta}\nu^{-1-\frac{\delta}{3}}c_{\lambda,\nu}.
\end{align}

The remaining case is $\lambda_1\ll \lambda_2\sim \lambda$. Here we use \eqref{4.29} to interpolate the bounds \eqref{BA1} and \eqref{BA2} for $E(u_{\lambda_2,\nu_2})$. Setting $s = \frac{3\delta}{4(2+\delta)}$, we obtain
\begin{align}
  &~ E(u_{\lambda_2,\nu_2})^{\frac{s}{2}}E(u_{\lambda_2})^{\frac{1-s}{2}}\lesssim (CN\nu_2^{-2-\delta}\lambda_2^{-\frac{1}{2}-\delta}c_{\lambda_2,\nu_2})^{\frac{3\delta}{4(2+\delta)}}(CM \lambda^{-\frac{5}{2}-3\delta}c_{\lambda_2})^{1-\frac{3\delta}{4(2+\delta)}}\nonumber\\
  \lesssim&~ C M^{1-\frac{3\delta}{4(2+\delta)}}N^{\frac{3\delta}{4(2+\delta)}}\nu_2^{-\frac{3\delta}{4}}\lambda_2^{-\frac{5}{2}-2\delta}c_{\lambda_2,\nu_2}c_{\lambda_2}\nonumber\\
  \lesssim&~ C M^{1-\frac{3\delta}{4(2+\delta)}}N^{\frac{3\delta}{4(2+\delta)}}\nu_2^{-\frac{2\delta}{3}}\lambda_2^{-\frac{5}{2}-2\delta}c_{\lambda_2,\nu_2}.
\end{align}
The last line follows from the envelope relation $\nu^{-\delta_0^2}c_\lambda\le c_{\lambda,\nu}$.
Together with the bound \eqref{BA2} of $E(u_{\lambda_1,\nu_1})$, we have 
\begin{align}
  &~ \sum_{\lambda_1\ll\lambda_2\sim \lambda}\sum_{\nu_{\text{max}}\gtrsim \nu}\sup_{y_\alpha}\|((\pa u_{\lambda,\nu})^{y_1} + r_{y_0}^{-1} (u_{\lambda,\nu})^{y_1})\cdot [R_\nu ,h^{i\alpha\beta}\pa_\beta u_{\lambda_1,\nu_1} \pa_i \pa_\alpha ] u_{\lambda_2,\nu_2}\|_{L^1_{t,x}}\notag\\
  \lesssim&~ E(u_{\lambda,\nu})^{\frac{1}{2}}\sum_{\lambda_1\ll\lambda_2\sim \lambda}\sum_{\nu_1\gtrsim \nu,\nu_2}\nu_{2}^{1+\delta_0}\lambda_{\text{min}}^{\frac{1}{2}+\delta_0}\lambda_2 E( u_{\lambda_1,\nu_1})^{\frac{1}{2}}E(u_{\lambda_2,\nu_2})^{\frac{s}{2}}E(u_{\lambda_2})^{\frac{1-s}{2}}\nonumber\\
  &  + E(u_{\lambda,\nu})^{\frac{1}{2}}\sum_{\lambda_1\ll\lambda_2\sim \lambda}\sum_{\nu_1\ll\nu_2\sim \nu}  \nu^{-1} \nu_{1}^{2+\delta_0} \lambda_{\rm{\min}}^{\frac{1}{2}+\delta_0} \lambda_2 E( u_{\lambda_1,\nu_1})^{\frac{1}{2}}E(u_{\lambda_2,\nu_2})^{\frac{s}{2}}E(u_{\lambda_2})^{\frac{1-s}{2}}\nonumber\\
  \lesssim&~ E(u_{\lambda,\nu})^{\frac{1}{2}}C^2 M^{1-\frac{3\delta}{4(2+\delta)}}N^{1+\frac{3\delta}{4(2+\delta)}} \Big(\sum_{\lambda_1\ll\lambda_2\sim \lambda}\sum_{\nu_1\gtrsim \nu,\nu_2 } \nu_2^{1+\delta_0 -\frac{2\delta}{3}}\nu_1^{-2-\delta}\lambda_1^{\delta_0-\delta}\lambda_2^{-\frac{3}{2}-2\delta}c_{\lambda_1,\nu_1}c_{\lambda_2,\nu_2}\nonumber\\
  & + \sum_{\lambda_1\ll\lambda_2\sim \lambda}\sum_{\nu_1\ll\nu_2\sim \nu}  \nu^{-1}\nu_1^{\delta_0-\delta}\lambda_1^{\delta_0-\delta}\nu_2^{-\frac{2\delta}{3}}\lambda_2^{-\frac{3}{2}-2\delta}c_{\lambda_1,\nu_1}c_{\lambda_2,\nu_2}\Big)\nonumber\\
  \lesssim&~ E(u_{\lambda,\nu})^{\frac{1}{2}} C^2\varepsilon_0\varepsilon_1\nu^{-1-\frac{\delta}{3}}\lambda^{-\frac{3}{2}-2\delta}\Big(\sum_{\lambda_2\sim \lambda}\Big(\nu^{-\frac{\delta}{3}}\sum_{\nu_2\lesssim\nu} \big(\frac{\nu}{\nu_2}\big)^{\delta_0^2}c_{\lambda_2,\nu} + \sum_{\nu_2\gtrsim \nu}\nu_2^{-\frac{\delta}{3}}c_{\lambda_2,\nu_2}\Big)+ \sum_{\lambda_2\sim \lambda}\sum_{\nu_2\sim \nu}c_{\lambda_2,\nu_2}\Big)\nonumber\\
  \lesssim&~ C \varepsilon_1E(u_{\lambda,\nu})^{\frac{1}{2}}\cdot C\varepsilon_0\nu^{-1-\frac{\delta}{3}}\lambda^{-\frac{3}{2}-2\delta}c_{\lambda,\nu}.
\end{align} 

\textbf{(C). Estimates for \eqref{u3} and \eqref{u4}. } We claim that the estimates in the present case are identical to the previous ones. We write 
\begin{equation}
  v_{\lambda} := \pa_t u_\lambda,\ v_{\lambda,\nu} := \pa_t u_{\lambda,\nu}.
\end{equation} 

Up to a translation, we write the source terms as   
\begin{align}
&~ |\mathbf{F}_e(v^{y_0}_\lambda)|,|\mathbf{F}_m(v^{y_0}_\lambda)| \label{non}\\
\sim&~  |(\pa v_\lambda + r_{y_0}^{-1} v_\lambda)\cdot [P_\lambda,\bar{h}^{\alpha \beta} ]\pa_\alpha \pa_\beta v|  + |(\pa v_\lambda + r_{y_0}^{-1} v_\lambda)\cdot G(\pa u,v_\lambda)| , \notag\\
& + |(\pa v_\lambda + r_{y_0}^{-1} v_\lambda)\cdot P_\lambda(\pa_t\bar{h}^{\alpha \beta}\pa_{\alpha}\pa_{\beta} u)|,\notag  
\end{align}
and 
\begin{align}
  &~|\mathbf{F}_e((v_{\lambda,\nu})^{y_0})|,|\mathbf{F}_m((v_{\lambda,\nu})^{y_0})|\\
   \sim&~  |(\pa v_{\lambda,\nu}+ r_{y_0}^{-1} v_{\lambda,\nu})\cdot [R_\nu P_\lambda,\bar{h}^{\alpha \beta} \pa_\alpha \pa_\beta] v|    +  |(\pa v_{\lambda,\nu} + r_{y_0}^{-1} v_{\lambda,\nu})\cdot G(\pa u,v_{\lambda,\nu})| \notag\\
  &~   + |(\pa v_{\lambda,\nu} + r_{y_0}^{-1} v_{\lambda,\nu})\cdot R_\nu P_\lambda(\pa_t\bar{h}^{\alpha \beta}\pa_{\alpha}\pa_{\beta} u)|,\notag
\end{align}
Therefore, the estimates for \eqref{u3} and \eqref{u4} are identical to those for \eqref{u1} and \eqref{u2}, except for the following low-order terms
\begin{align}
  &~ |(\pa v_\lambda + r_{y_0}^{-1} v_\lambda)\cdot P_\lambda(\pa_t\bar{h}^{\alpha \beta}\pa_{\alpha}\pa_{\beta} u)| \sim  \sum_{\lambda_{\max}\sim  \lambda_{\text{med}}} |(\pa v_\lambda + r_{y_0}^{-1} v_\lambda)\cdot G(\pa_x u_{\lambda_1}, \pa_t u_{\lambda_2})|,\\
  &~ |((\pa v_{\lambda,\nu}) + r_{y_0}^{-1} v_{\lambda,\nu})\cdot R_\nu P_\lambda(\pa_t\bar{h}^{\alpha \beta}\pa_{\alpha}\pa_{\beta} u)|\notag\\
  &~ \quad\quad\quad \sim \sum_{\lambda_{\max}\sim  \lambda_{\text{med}}}\sum_{\nu_{\max}\sim  \nu_{\text{med}}} |(\pa v_{\lambda,\nu} + r_{y_0}^{-1} v_{\lambda,\nu})\cdot G(\pa_x u_{\lambda_1,\nu_1}, \pa_t u_{\lambda_2,\nu_2})|,
\end{align}
Here we used $\pa_t\bar{h}^{\alpha\beta}=-h^{i\alpha\beta}\pa_t\pa_\beta u$ and \eqref{r1} to reduce these terms to the null form $G$; they are the only new contributions compared with \eqref{u1} and \eqref{u2}.
For $v_\lambda$ we have the following bounds:
\begin{align}
&~ \sum_{\lambda_{\max}\sim \lambda_{\text{med}}} \|(\pa v_\lambda + r_{y_0}^{-1} v_\lambda)\cdot G(\pa_x u_{\lambda_1}, \pa_t u_{\lambda_2})\|_{L^1_{t,x}}\notag\\
\lesssim&~ \sum_{\lambda_1\gtrsim \lambda} \sum_{\lambda_2,\nu_2} E(\pa_t u_{\lambda,\nu})^{\frac{1}{2}} (\lambda_{2}^{\frac{1}{2}+\delta_0}\nu_{2}^{1+\delta_0})\lambda_1 E(u_{\lambda_1})^{\frac{1}{2}}E(\pa_t u_{\lambda_2,\nu_2})^{\frac{1}{2}}\notag\\
&   + \sum_{\lambda_2\gtrsim \lambda} \sum_{\lambda_1,\nu_1} E(\pa_t u_{\lambda,\nu})^{\frac{1}{2}} (\lambda_{1}^{\frac{1}{2}+\delta_0}\nu_{1}^{1+\delta_0})\lambda_1 E(u_{\lambda_1,\nu_1})^{\frac{1}{2}}E(\pa_t u_{\lambda_2})^{\frac{1}{2}}\notag\\
\lesssim&~ C\varepsilon_0 E(\pa_t u_{\lambda,\nu})^{\frac{1}{2}} \sum_{\lambda'\gtrsim \lambda}(\lambda' E(u_{\lambda'})^{\frac{1}{2}}+E(\pa_t u_{\lambda'})^{\frac{1}{2}}),
\end{align}
where we apply \eqref{refined} to the low frequency terms. 

For $v_{\lambda,\nu}$ we proceed similarly:
\begin{align}\label{S3}
  &~ \sum_{\lambda_{\max}\sim \lambda_{\text{med}}}\sum_{\nu_{\max}\sim \nu_{\text{med}}} \|(\pa v_{\lambda,\nu} + r_{y_0}^{-1} v_{\lambda,\nu})\cdot G(\pa_x u_{\lambda_1,\nu_1}, \pa_t u_{\lambda_2,\nu_2})\|_{L^1_{t,x}}\notag\\
  \lesssim&~ \sum_{\lambda_{\max}\sim \lambda_{\text{med}}}\sum_{\nu_{\max}\sim \nu_{\text{med}}}\nu_{\text{min}}^{1+\delta_0}\lambda_{\text{min}}^{\frac{1}{2}+\delta_0}\lambda_1 E(\pa_t u_{\lambda,\nu})^{\frac{1}{2}} E(u_{\lambda_1,\nu_1})^{\frac{1}{2}} E(\pa_t u_{\lambda_2,\nu_2})^{\frac{1}{2}}.
\end{align}
Here we used \eqref{4.34} with the outer factor $(\pa v_{\lambda,\nu}+r_{y_0}^{-1}v_{\lambda,\nu})$ and \eqref{re} for $\pa_x u_{\lambda_1,\nu_1}$.
We first consider the case $\lambda_2 \gtrsim \lambda, \lambda_1$. If $\nu_1 \ll \nu_2\sim \nu$, we directly apply \eqref{refined} to $E(u_{\lambda_1,\nu_1})$ and  $E(\pa_t u_{\lambda_2,\nu_2})$. If $\nu_1\gtrsim \nu,\nu_2$, we apply \eqref{refined2} to $E(u_{\lambda_1,\nu_1})$ and \eqref{refined} to $E(\pa_t u_{\lambda_2,\nu_2})$. Then the summation in \eqref{S3}  is bounded by
\begin{align}
  &~ E(\pa_t u_{\lambda,\nu})^{\frac{1}{2}} \sum_{\lambda_2 \gtrsim \lambda, \lambda_1}\sum_{\nu_1 \ll \nu_2\sim \nu}\nu_{1}^{1+\delta_0}\lambda_{1}^{\frac{3}{2}+\delta_0}E(u_{\lambda_1,\nu_1})^{\frac{1}{2}}   E(\pa_t u_{\lambda_2,\nu_2})^{\frac{1}{2}}\nonumber\\
  &  +E(\pa_t u_{\lambda,\nu})^{\frac{1}{2}}\sum_{\lambda_2 \gtrsim \lambda, \lambda_1}\sum_{\nu_1 \gtrsim \nu, \nu_2}\nu_{2}^{1+\delta_0}\lambda_{1}^{\frac{3}{2}+\delta_0}   E(u_{\lambda_1,\nu_1})^{\frac{1}{2}} E(\pa_t u_{\lambda_2,\nu_2})^{\frac{1}{2}}\nonumber\\
  \lesssim&~ E(\pa_t u_{\lambda,\nu})^{\frac{1}{2}} C\varepsilon_0 \sum_{\lambda_2\gtrsim \lambda}\sum_{\nu_2\sim \nu}E(\pa_t u_{\lambda_2,\nu_2})^{\frac{1}{2}}\nonumber\\
  &  +E(\pa_t u_{\lambda,\nu})^{\frac{1}{2}} C^2 \varepsilon_0\varepsilon_1\sum_{\lambda_2\gtrsim \lambda,\lambda_1}\sum_{\nu_1\gtrsim\nu,\nu_2} \nu_2^{\delta_0-\frac{\delta}{3}}\lambda_1^{\delta_0-\frac{6\delta}{5}}\nu_1^{-1-\delta}\lambda_2^{-\frac{1}{2}-\delta}c_{\lambda_1,\nu_1}c_{\lambda_2,\nu_2}\nonumber\\
  \lesssim&~ E(\pa_t u_{\lambda,\nu})^{\frac{1}{2}} C\varepsilon_0 \sum_{\lambda_2\gtrsim \lambda}\sum_{\nu_2\sim \nu}E(\pa_t u_{\lambda_2,\nu_2})^{\frac{1}{2}}\nonumber\\
  &  +E(\pa_t u_{\lambda,\nu})^{\frac{1}{2}} C^2 \varepsilon_0\varepsilon_1\sum_{\lambda_2\gtrsim \lambda,\lambda_1}\sum_{\nu_1\gtrsim\nu}\lambda_1^{\delta_0-\frac{6\delta}{5}}\nu_1^{-1-\delta}\lambda_2^{-\frac{1}{2}-\delta}c_{\lambda_1,\nu_1}\Big(\sum_{\nu_2\gtrsim \nu} \nu_2^{\delta_0-\frac{\delta}{3}}c_{\lambda_2,\nu_2}+ \sum_{\nu_2\ll\nu} \big(\frac{\nu}{\nu_2}\big)^{\delta_0^2}c_{\lambda_2,\nu}\Big)\nonumber\\
  \lesssim&~ E(\pa_t u_{\lambda,\nu})^{\frac{1}{2}} C\varepsilon_0 \sum_{\lambda_2\gtrsim \lambda}\sum_{\nu_2\sim \nu}E(\pa_t u_{\lambda_2,\nu_2})^{\frac{1}{2}} + C \varepsilon_1  E(\pa_t u_{\lambda,\nu})^{\frac{1}{2}} \cdot C\varepsilon_0 \lambda^{-\frac{1}{2}-\delta} \nu^{-1-\frac{\delta}{3}} c_{\lambda,\nu},
\end{align}
which is acceptable. 

Next we consider the case $\lambda_2\ll \lambda_1\sim \lambda$.  We directly apply \eqref{refined} to $E(u_{\lambda_1,\nu_1})$ and  $E(\pa_t u_{\lambda_2,\nu_2})$ and obtain
\begin{align}
  &~ E(\pa_t u_{\lambda,\nu})^{\frac{1}{2}} \sum_{\lambda_2 \ll \lambda_1\sim  \lambda}\sum_{\nu_1 \gtrsim \nu, \nu_2}\nu_{2}^{1+\delta_0}\lambda_1\lambda_{2}^{\frac{1}{2}+\delta_0}   E(u_{\lambda_1,\nu_1})^{\frac{1}{2}} E(\pa_t u_{\lambda_2,\nu_2})^{\frac{1}{2}}\nonumber\\
  & +E(\pa_t u_{\lambda,\nu})^{\frac{1}{2}}\sum_{\lambda_2 \ll \lambda_1\sim  \lambda}\sum_{\nu_1 \ll \nu_2\sim \nu}\nu_{1}^{1+\delta_0}\lambda_1\lambda_{2}^{\frac{1}{2}+\delta_0}E(u_{\lambda_1,\nu_1})^{\frac{1}{2}}   E(\pa_t u_{\lambda_2,\nu_2})^{\frac{1}{2}}\nonumber\\
  \lesssim&~ E(\pa_t u_{\lambda,\nu})^{\frac{1}{2}} C\varepsilon_0 \sum_{\lambda_1\sim  \lambda}\sum_{\nu_1\gtrsim \nu}\lambda_1 E( u_{\lambda_1,\nu_1})^{\frac{1}{2}}\nonumber\\
  & +E(\pa_t u_{\lambda,\nu})^{\frac{1}{2}} C^2 \varepsilon_0\varepsilon_1\sum_{\lambda_2 \ll \lambda_1\sim  \lambda}\sum_{\nu_1\ll\nu_2\sim \nu} \nu_1^{\delta_0-\frac{\delta}{3}}\nu_2^{-1-\frac{\delta}{3}}\lambda_1^{-\frac{1}{2}-2\delta}\lambda_2^{\delta_0-\delta}c_{\lambda_1,\nu_1}\cdot \big(\frac{\lambda}{\lambda_2}\big)^{\delta_0^2}c_{\lambda,\nu_2}\nonumber\\
  \lesssim&~ E(\pa_t u_{\lambda,\nu})^{\frac{1}{2}} C\varepsilon_0 \sum_{\lambda_1\sim  \lambda}\sum_{\nu_1\gtrsim \nu}\lambda_1 E( u_{\lambda_1,\nu_1})^{\frac{1}{2}}+E(\pa_t u_{\lambda,\nu})^{\frac{1}{2}} C^2 \varepsilon_0\varepsilon_1\lambda^{-\frac{1}{2}-\delta}\sum_{ \nu_2\sim \nu} \nu_2^{-1-\frac{\delta}{3}}   c_{\lambda ,\nu_2}\nonumber\\
  \lesssim&~ E(\pa_t u_{\lambda,\nu})^{\frac{1}{2}} C\varepsilon_0 \sum_{\lambda_2\gtrsim \lambda}E(\pa_t u_{\lambda_2,\nu_2})^{\frac{1}{2}} + C \varepsilon_1  E(\pa_t u_{\lambda,\nu})^{\frac{1}{2}} \cdot C\varepsilon_0 \lambda^{-\frac{1}{2}-\delta} \nu^{-1-\frac{\delta}{3}} c_{\lambda,\nu},
\end{align}
which is also acceptable.

Combining the above estimates, we complete the proof.
\end{proof}

Now we are ready to prove Theorem \ref{bound}. By the definition of $E(v)$ and Proposition \ref{Source}, we have 
\begin{align}
  &~ E(u_\lambda)^{\frac{1}{2}},\lambda^{-1}E(\pa_t u_\lambda)^{\frac{1}{2}}\lesssim (C^{\frac{3}{2}}\varepsilon^{\frac{1}{2}}+1)  M  \lambda^{-\frac{5}{2}-3\delta}c_\lambda, \\
  &~ E(u_{\lambda,\nu})^{\frac{1}{2}} \lesssim (C^{\frac{3}{2}}\varepsilon^{\frac{1}{2}}+1)  N  \lambda^{-\frac{1}{2}-\delta} \nu^{-2-\delta}c_{\lambda,\nu},\\
  &~ E(u_{\lambda,\nu})^{\frac{1}{2}}\lesssim (C^{\frac{3}{2}}\varepsilon^{\frac{1}{2}}+1) \varepsilon_0 \lambda^{-\frac{3}{2}-2\delta}\nu^{-1-\frac{\delta}{3}} c_{\lambda,\nu},\\
  &~ \lambda^{-1}E(\pa_t  u_{\lambda,\nu})^{\frac{1}{2}}\lesssim (C^{\frac{3}{2}}\varepsilon^{\frac{1}{2}}+1) \varepsilon_0 \lambda^{-\frac{3}{2}-\delta}\nu^{-1-\frac{\delta}{3}} c_{\lambda,\nu}.
\end{align}
Now we choose $C =\varepsilon^{-\frac{1}{8}}$. It is clear that such a constant meets the smallness assumption \eqref{sm}, since $C\varepsilon = \varepsilon^{\frac{7}{8}}\ll 1 $. Moreover, we have  
\begin{equation}
  (C^{\frac{3}{2}}\varepsilon^{\frac{1}{2}}+1)\lesssim 1 \ll C .
\end{equation}
This improves the constants in \eqref{BA1}, \eqref{BA2} and \eqref{refined}, which closes the bootstrap argument.
\subsection{Proof of Theorem \ref{main}} Since $s>s_c +1$, the local well-posedness of the Cauchy problem \eqref{quasi} is ensured by the classical theory. We can then apply Theorem \ref{bound} and extend this unique local solution to a unique global solution.

It remains to prove the Lipschitz dependence \eqref{lip}. Given two different solutions $u$ and $w$ as in Theorem \ref{main}, we have the difference equation 
\begin{equation}\begin{cases}
 \Box (u-w) = h^{i \alpha \beta }\pa_i\pa_{\alpha}(u-w)\cdot\pa_{\beta} u  + h^{i \alpha \beta }\pa_i\pa_{\alpha}w\cdot\pa_{\beta} (u-w), &\text{in} \mathbb{R}_+ \times \mathbb{R}^3\\
 \pa(u-w)[0] = \pa u[0]- \pa w[0],  & \text{ on}\ \mathbb{R}^3, \ \\
\end{cases}
\end{equation}
Its initial data also satisfy the condition in Theorem \ref{main}. 
The equation for $(u-w)_\lambda$ reads
\begin{equation}
  \Box_{\bar{g}} (u-w)  = [P_\lambda, \bar{h}^{\alpha \beta}\pa_\alpha \pa_\beta](u-w) + P_{\lambda} G(\pa_x w, u-w),
\end{equation}
Here we used the exact reduction $\Box_{\bar{g}}(u-w)=h^{i\alpha\beta}\pa_i\pa_\alpha w\pa_\beta(u-w)=G(\pa_x w,u-w)$.
It therefore suffices to control the energy functional $E((u-w)_\lambda)$. We again make the following bootstrap assumption:
\begin{equation}
  E((u-w)_\lambda)^{\frac{1}{2}} \le  C M' \lambda^{-\frac{3}{2}-3\delta} d_\lambda ,
\end{equation}
where 
\begin{equation}
  \begin{aligned}
     &~ M' = \|\pa u[0]-\pa w[0]\|_{H^{\frac{3}{2}+3\delta}},\\
     &~ d_\lambda = \sup_{\lambda'} e^{-\delta^2  |\ln \lambda- \ln \lambda'|}  (M')^{-1}\| P_{\lambda'} \pa u[0]-\pa w[0]\|_{H^{\frac{3}{2}+3\delta}}  + c_\lambda +\big(\sum_{\nu} c_{\lambda,\nu}^2\big)^{\frac{1}{2}}.\\
   \end{aligned}
\end{equation}

Up to a translation, we rewrite the nonlinearity as 
\begin{align}
  &~ |\mathbf{F}_e [(u-w)_\lambda^{y_0}]|, |\mathbf{F}_m [(u-w)_\lambda^{y_0}]|\\
   \sim&~ |(\pa (u-w)_\lambda + r_{y_0}^{-1} (u-w)_\lambda)\cdot [P_\lambda,\bar{h}^{\alpha \beta} ]\pa_\alpha \pa_\beta (u-w)| \notag\\
&  + |(\pa (u-w)_\lambda + r_{y_0}^{-1} (u-w)_\lambda)\cdot G(\pa u,(u-w)_\lambda)|  \notag\\
&  + |(\pa (u-w)_\lambda + r_{y_0}^{-1} (u-w)_\lambda)\cdot P_\lambda G(\pa_x w, u-w)|,\notag
\end{align}
which is almost the same as \eqref{non}. Therefore, repeating the procedure of proving \eqref{u3}, we have 
\begin{equation}
\|\mathbf{F}[(u-w)_\lambda]\|_{L^1_{t,x}}\lesssim  C \varepsilon (C M' \lambda^{-\frac{3}{2}-3\delta}d_\lambda)^2,
\end{equation}
which gives
\begin{equation}
  E((u-w)_\lambda)^{\frac{1}{2}}  \le (C^{\frac{3}{2}}\varepsilon^{\frac{1}{2}}+1)  M'  \lambda^{-\frac{3}{2}-3\delta}d_\lambda.
\end{equation}
Choosing $C$ as before, we complete the proof of the Lipschitz continuous dependence.
\appendix
\section{Proof of Lemma \ref{ancom}}
This appendix is devoted to the proof of the $R_\nu$-commutator estimate, which essentially repeats the arguments in \cite[Appendix A.1]{guo_Global_2023}.

Denoting by $P_n^{(a,b)}$ the Jacobi polynomials, recall that 
\begin{equation}
 P^{(0,0)}_n(x)=L_n(z)=\frac{1}{2^n n!}\frac{\mathrm{d}^n}{\mathrm{d}x^n}[(x^2-1)^n],
\end{equation}
and therefore
\begin{equation}\label{A2}
  (2n+1)P^{(0,0)}_n(x)=\frac{\mathrm{d}}{\mathrm{d}x}(P_{n+1}^{(0,0)}(x)-P_{n-1}^{(0,0)}(x) ).
\end{equation}
We recall the following differentiation formula (see \cite[Section 4.5]{szego_Orthogonal_1939})
\begin{equation}\label{DerJacobi}
 \frac{\mathrm{d}}{\mathrm{d}x}P^{(a,a)}_n=\frac{n+2a+1}{2}P^{(a+1,a+1)}_{n-1},\quad \forall a\in \mathbb{N}.
\end{equation}

Moreover, we have the following asymptotics (see \cite[Section 8.21]{szego_Orthogonal_1939}):
\begin{lemma}\label{LemZonal0}
Fix $0<c<\pi$, then 
\begin{equation}\label{SogForm}
  \begin{aligned}
 &~ P^{(a,a)}_n(\cos\theta)= \\
  &\begin{cases}2^a\sqrt{\frac{2}{\pi}}\frac{1}{\sqrt{n}(\sin\theta)^{a+\frac{1}{2}}}\left(\cos((n+a+\frac{1}{2})\theta-\frac{(2a+1)\pi}{4})+\frac{1}{n\sin\theta}O(1)\right),& c/n\le\theta\le\pi-c/n,\\
&~ O(n^a), \text{else}.
\end{cases}
   \end{aligned}
\end{equation}
\end{lemma}

For \eqref{angular}, without loss of generality we assume $f,g$ are functions on $\mathbb{S}^2$, and we have 
\begin{align*}
  &~ [R_\nu, f]g(\omega_1) =  \sum_{n\ge 0}\varphi(\nu^{-1}n)\int_{\mathbb{S}^2} [f(\omega_1)-f(\omega_2)]g(\omega_2)\mathfrak{Z}_n(\langle \omega_1,\omega_2\rangle)\mathrm{d}\mu_{\mathbb{S}^2}(\omega_2)\\
  =&~ \int_{0}^1 \int_{\mathbb{S}^2}\pa_\theta f(\omega(t))\cdot \theta(\omega_1,\omega_2)g(\omega_2)\sum_{n\ge 0}\varphi(\nu^{-1}n)\mathfrak{Z}_n(\langle \omega_1,\omega_2\rangle)\mathrm{d}\mu_{\mathbb{S}^2}(\omega_2)\mathrm{d}t,
\end{align*}
where $\varphi(\nu^{-1}n) = \chi(\nu^{-1}n)-\chi((2\nu)^{-1}n)$, $\{\omega(t)\}_{0\le t\le 1}$ is the shorter circle path connecting $\omega_1$ and $\omega_2$ and $\theta(\omega_1,\omega_2)$ is the geodesic distance on $\mathbb{S}^2$.

Denote
\begin{equation}
K_\nu(\omega_1,\omega_2)=\sum_{n\ge 0}\varphi(\nu^{-1}n) \theta(\omega_1,\omega_2)\mathfrak{Z}_n(\langle \omega_1,\omega_2\rangle),
\end{equation}
and to prove \eqref{angular} it suffices to prove
\begin{equation}
\sup_{\omega_1}\Vert K_\nu(\omega_1,\omega_2)\Vert_{L^1_{\omega_2}}+\sup_{\omega_2}\Vert K_\nu(\omega_1,\omega_2)\Vert_{L^1_{\omega_1}}\lesssim \frac{1}{\nu}.
\end{equation}
This essentially follows from \eqref{SogForm}. With \eqref{A2} and \eqref{DerJacobi} we have
\begin{align*}
 &~ (2n+1)P_n^{(0,0)}(x) =\frac{1}{2n+3}\frac{\mathrm{d}^2}{\mathrm{d}x^2}P_{n+2}^{(0,0)}(x) + \frac{1}{2n-1}\frac{\mathrm{d}^2}{\mathrm{d}x^2}P_{n-2}^{(0,0)}(x)\\
 &~ \quad - \Big(\frac{1}{2n+3}+\frac{1}{2n-1}\Big)\frac{\mathrm{d}^2}{\mathrm{d}x^2}P_n^{(0,0)}(x)\\
 =&~ \frac{Q_{n}}{2n+3}P_{n}^{(2,2)}(x) + \frac{Q_{n-4}}{2n-1} P_{n-4}^{(2,2)}(x)- \Big(\frac{Q_{n-2}}{2n+3}+\frac{Q_{n-2}}{2n-1}\Big) P_{n-2}^{(2,2)}(x),\\
 &~ \text{where}\ Q_{n} = \frac{(n+3)(n+4)}{4},
\end{align*}
and therefore
\begin{align*}
 &~ \sum_{n\geq 0}\varphi(\nu^{-1}n)\mathfrak{Z}_n(x) = \sum_{n\geq 0}\varphi(\nu^{-1}n)\frac{2n+1}{4\pi}P_n^{(0,0)}(x)=\frac{1}{8\pi}\sum_{n\geq 0}P_n^{(2,2)}(x)\cdot D_n,\\
  &~ \text{where}\  D_n:=Q_n\Big[\frac{\varphi(\nu^{-1}(n+4))-\varphi(\nu^{-1}(n+2))}{2n+7}-\frac{\varphi(\nu^{-1}(n+2))-\varphi(\nu^{-1}(n))}{2n+3}\Big].
\end{align*}
In the region $\{\frac{c}{\nu}\le  \theta\le \pi-\frac{c}{\nu}\}$, in view of \eqref{SogForm}, we set
\begin{align*}
&~ C_n(\theta)=\sum_{0\le j\le n-1}\cos\Big((j+\frac{3}{2})\theta-\frac{3\pi}{4}\Big)=\frac{\sin\frac{n\theta}{2}}{\sin\frac{\theta}{2}}\cos\Big(\left(\frac{n}{2}+1\right)\theta-\frac{3\pi}{4}\Big),\\
&~ \big|\sum_{n\geq 0}P_n^{(2,2)}(\cos \theta)\cdot D_n\big|\lesssim |I_1(\theta)|+ |I_{2}(\theta)|,
\end{align*}
where
\begin{align*}
  &~ I_1(\theta) :=(\sin\theta)^{-5/2}\sum_{n\ge0}\frac{1}{\sqrt{n}}\cos((n+\frac{3}{2})\theta-\frac{3\pi}{4})\cdot D_n\\
  &~ =(\sin\theta)^{-5/2}\cdot \sum_{n\ge0}\frac{D_n}{\sqrt{n}}\cdot\left[C_{n+1}(\theta)-C_n(\theta)\right]\\
  &~ =(\sin\theta)^{-5/2}\cdot\sum_{n\ge0}C_n(\theta)\left[\frac{D_{n-1}}{\sqrt{n-1}}-\frac{D_n}{\sqrt{n}}\right]\\
  &~ I_2(\theta) :=(\sin\theta)^{-7/2}\sum_{n\ge0}\frac{|D_n|}{n^{3/2}}.
\end{align*}
Since $|{\frac{D_{n+1}}{\sqrt{n+1}}-\frac{D_{n}}{\sqrt{n}}}|\lesssim \nu^{-\frac{5}{2}},\ |D_n|\lesssim \nu^{-1}$ and ${C_n(\theta)}\lesssim {\sin(\frac{\theta}{2})}^{-1}$, we have
\begin{equation}
\int_{\frac{c}{\nu}\le \theta\le\pi-\frac{c}{\nu}}(\vert I_1(\theta)\vert+|I_2(\theta)|) \theta\sin\theta d\theta\lesssim \nu^{-1}.
\end{equation}
In the complement of the region $\{\frac{c}{\nu}\le  \theta\le \pi-\frac{c}{\nu}\}$, again by \eqref{SogForm}, we have 
\begin{equation}
  \sum_{n\geq 0}\varphi(\nu^{-1}n)\int_{0\le \theta\le \frac{c}{\nu}}{|P_n^{(2,2)}(\cos(\theta)) D_n |}\cdot \theta \sin\theta d\theta\lesssim \frac{1}{\nu}.
\end{equation}

Combining the above estimates, we complete the proof.

{\bf Funding declaration.} This work was supported by the National Natural Science Foundation of China [No. 12171097], the
Key Laboratory of Mathematics for Nonlinear Sciences (Fudan University), the Ministry of Education
of China, Shanghai Key Laboratory for Contemporary Applied Mathematics.

{\bf Data availability.} The authors confirm that the data supporting the findings of this study are available
within the article.

{\bf Conflict of interest.} The authors declare that this work does not have any conflicts of interest.

{\bf Ethics declaration.}  Not applicable.

\bibliographystyle{plainnat}
\bibliography{ref}

\end{document}